\documentclass[12pt]{article}
\usepackage{amsmath,amsthm,amssymb} 
\usepackage{enumerate}
\usepackage{latexsym}
\usepackage{graphicx}
\usepackage[usenames,dvipsnames]{xcolor}
\usepackage[margin=1in]{geometry}
\usepackage{float}
\usepackage{bbm}
\usepackage{relsize}
\allowdisplaybreaks
\def\wt{\widetilde}

\def\blue{\color{blue}}

\def\red{\color{red}}

\newcommand{\set}[1]{\left\{#1\right\}}
\def\cH{{\mathcal H}}

\def\cE{{\mathcal E}}
\def\cA{{\mathcal A}}

\def\whL{\widehat{L}}

\newcommand{\brac}[1]{\left( #1 \right)}
\newcommand\bfrac[2]{\left(\frac{#1}{#2}\right)}
\def\E{\mathbb{E}\,}

\def\Pr{\mathbb{P}}

\def\bZ{\mathbb{Z}}

\def\rai{\rightarrow \infty}

\newcommand{\ul}[1]{\boldsymbol{#1}}

\def\ooi{(1+o(1))}
\newcommand{\ind}[1]{\,\mathbbm{1}_{\{{#1}\}}}

\newcommand{\ignore}[1]{ }

\def\b{\beta}
\def\d{\delta}

\def\e{\varepsilon}
\def\f{\phi}

\def\g{\gamma}

\def\Th{\Theta}

\def\m{\mu}

\def\r{\rho}

\def\t{\tau}
\def\om{\omega}
\def\Om{\Omega}

\def\cB{{\mathcal B}}

\newtheorem{theorem}{Theorem}

\newtheorem{lemma}[theorem]{Lemma}

\newtheorem{proposition}[theorem]{Proposition}

\newcommand{\ra}{\rightarrow}
\newcommand{\raz}{\rightarrow 0}
\newcommand{\beq}[2]{\begin{equation}\label{#1}#2\end{equation}}

\def\sm{\! \setminus \!}

\newcounter{rot}

\def\seq{\subseteq}

\def\wh{\widehat}
\def\wt{\widetilde}
\def\whL{\wh L}
\def\wtL{\widetilde L}

\def\stL{L^*}

\newcommand{\rdup}[1]{ \lceil #1 \rceil}
\newcommand{\rdown}[1]{\lfloor #1 \rfloor}
\def\1{{\bf 1}}
\def\0{{\bf 0}}

\begin{document}
\title{ Diffusion limited aggregation in the layers model}

\author{Colin Cooper\thanks{
Research supported  at the University of Hamburg, by a  Mercator fellowship from DFG  Project 491453517}\\Department of Informatics\\
King's College\\
London WC2B 4BG\\England
\and Alan Frieze\thanks{Research supported in part by NSF Grant  DMS1952285}\\
Department of Mathematical Sciences\\
Carnegie Mellon University\\
Pittsburgh PA15213\\
U.S.A.}

\maketitle

\begin{abstract}
In the classical model of Diffusion Limited Aggregation (DLA), introduced by
Witten and Sander, the process begins with a single particle cluster placed
at the origin of a space. Then, one at a time, particles make a random
walk from infinity until they halt by colliding with the existing
cluster.

We consider an analogous version of this process on large but finite graphs
with a designated source and sink vertex. Initially the cluster of halted particles contains a
single particle at the sink vertex. Starting one at a time from the source, each
particle makes a random walk in the direction of the sink vertex.
The particle halts at the last unoccupied vertex  before the walk enters the cluster for the first time,
thus
increasing the size of the cluster. This continues until the source vertex becomes occupied,
at which point the process ends.

We study this DLA process on  several classes of layered graphs, including Cayley trees   of   branching factor  at least two with a sink
vertex attached to the leaves.  We
determine the finish time of the process for the given classes of graphs and show that the subcomponent of
the final cluster linking source to sink is essentially a unique path.
\end{abstract}

\section{Introduction}

\paragraph{Diffusion limited aggregation.}

In the classical model of Diffusion Limited Aggregation (DLA), introduced by Witten and Sander  \cite{WS}, \cite{WS1}, the process begins with a single particle
cluster placed at the origin of a space, and then, one-at-a-time,  particles make a random
walk “from infinity” until they collide with and stick to the existing cluster.
 The process is particularly natural in Euclidean space with particles making Brownian motion, or on the $d$--dimensional lattice $\bZ^d$.
Simulations of DLA in two dimensions show tree-like figures with long branches.  For $\bZ^d$, Kesten \cite{Kesten} proved  that when the cluster size is $N$, the length of these arms   is almost surely upper bounded by $CN^{2/3}$, when $d = 2$, and by $C_d N^{2/d}$ when $d\ge 3$.

\ignore{
Many similar  processes have been considered, the simplest, Reaction Limited Aggregation (RLA), differs from DLA in that  particles stick with probability less than one on each collision.
Indicative publications on RLA include  Ball et al \cite{Ball} or Meakin and Family \cite{Meakin}.
}

A distinct but  related process, Internal Diffusion Limited Aggregation (IDLA), was introduced by Diaconis
and Fulton \cite{Diaconis}, as a protocol for recursively building a random aggregate
of particles.
In IDLA
particles are added to the source vertex of an infinite graph, and make a random walk (over  occupied vertices) until they visit an unoccupied vertex at which point they halt. Thus the first particle occupies the source, and subsequent particles stick to the outside of the  component rooted at the source.
{Although the DLA and IDLA processes differ, the point in common is that they both describe the evolution of a unique cluster by adhesion to the cluster boundary.}
As with DLA, the focus in IDLA has been on the limiting shape of the component formed by the occupied vertices. The formative work by Lawler et al \cite{Lawler},  proved that, on $d$-dimensional lattices the limiting shape approaches a Euclidean ball; a result subsequently refined in \cite{Asselah},  \cite{Jerison} and  \cite{Lawler1}, amongst others.

DLA has been proposed as a model of physical processes in systems as diverse as coagulated aerosols \cite{WS}  and urban growth \cite{BLF}, a common factor being the dendritic shape of the cluster so obtained (see \cite{HAL1} for  illustrative examples).
However the main and an original motivation for the DLA process was as a model of dendritic growth in dielectric breakdown and lightning formation.  Niemeyer, Pietronero, and Weismann, \cite{NPW}, introduced a dielectric breakdown model which considers
DLA in the presence of an  electric field which biasses the particles to move in a given direction (e.g. downward).
In  established models of lightning formation,
 paths of negatively charged particles (leaders) descend downward from the
 base of the cloud layer, and paths of positively charged particles
 percolate upwards from the ground towards them, inducing a lightning strike  on meeting.
The DLA process was seen as a first approximation of this  process, which is itself finite in extent with a source at the top (cloud layer) and a sink (the earth) at the bottom.


We consider a version of the DLA process on large but finite graphs with a
designated source and sink vertex.
DLA on finite graphs
was previously studied   for complete binary trees by Hastings and Halsey \cite{HH}, and for the Boolean lattice (hypercube) 
by Frieze and Pegden \cite{FP}. The current paper continues this analysis of DLA on finite layered graphs, of which the binary tree and hypercube are typical examples.

\ignore{
The process in \cite{FP}  evolves at discrete time steps $t=0,1,\dots$; each of which has an associated cluster $C_t$.  The initial cluster $C_0$ consists of  a  single occupied vertex $\0=(0,\dots,0)$, the sink vertex.  Let $\1=(1,\dots,1)$ be the source vertex of the particles. We regard the edges of $\cB$ as directed away from $\1$ towards $\0$. The cluster $C_t$ is produced from $C_{t-1}$ by choosing a random  walk $\r_t$ from the source vertex $\1$ to $\0$ on this directed graph. Let $v$ be the last vertex of the  initial segment of $\r_t$ which is disjoint from $C_{t-1}$, and set $C_t=C_{t-1}\cup \{v\}$.  The process terminates at the first time $t_f$ when $\1 \in C_{t_f}$.

{\red {\sc what about this?} We consider two versions of the process. {\em Sticky} and {\em Non-Sticky}. In the Sticky process the particle continues until it reaches a vertex with an occupied child (the Physicists model) and in the Non-Sticky process the particle only stops when it tries to occupy an already occupied vertex. The boundary model.
Are they approx the same in terms of finish time? for finite graphs I think so.
}
{\blue At this point we distinguish two related DLA processes. In the examples from the physics literature cited above the particle halts on reaching a vertex with a neighbour in the occupied cluster, thus adhering (sticking) to the cluster boundary, the {\em Sticky}  process. In the alternative {\em Non-Sticky} process studied in \cite{FP} the particle only halts when the next transition is to an occupied vertex, i.e., into the occupied cluster.}
}

\paragraph{The layers model.}
{
Let $G$ be a finite graph $G=(V,E)$, whose vertices can be partitioned into sets $S_0,S_1,\ldots,S_k,S_{k+1}$ to form a
{\em layered structure} in which all edges are 
between layer $S_i$ and $S_{i+1}$ ($i=0,1,...,k$). The sets $S_0,\,S_{k+1}$ are  of size one with $S_0=\{v\}$ and $S_{k+1}=\{z\}$. The vertex $v$ is the {\em source vertex} and the vertex $z$ is the {\em sink vertex}.
In certain cases the source and sink may be attached as extra ({\em artificial}) vertices to an existing graph $G$ to complete the layered structure.

Examples of layered graphs with a symmetric structure include the following: The Boolean lattice (hypercube)  $\cB=\{0,1\}^n$, with $v=\1$, $z=\0$. Two dimensional square grids  with  source the top left hand corner and sink the bottom right hand corner, thus inducing a diagonal orientation. Finite Cayley trees with source the root and an artificial sink. Layered multipartite graphs with attached source and sink.

Graphs in the layers model have a top (the source) and a bottom (the sink).
Particles are constrained to move  downwards.
As such they can be seen as simple models of
particles percolating downward through some porous medium.
}

\paragraph{DLA in the layers model.}

Initially at step $t=0$, only the sink vertex  $z$ is  occupied, and  the occupied component is $C_0=\{z\}$.
At  a given step $t$, let $\r_t$ be a random walk  on the underlying graph  starting at the source $v$
and moving forward  level  by level
 to the sink $z$.   A particle  placed on the source vertex (assumed unoccupied) follows the walk $\r_t$ until it encounters an occupied vertex.
Let the path followed by the walk $\r_t$ be $v=x_0x_1x_2,...,x_kx_{k+1}=z$, and let 
$x_{i}$, ($i \ge 1$) be the first occupied vertex on the walk.
The particle halts at  position  $x_{i-1}$  and permanently occupies that vertex.
The component $C_{t}$ is formed by
 adding the vertex $x_{i-1}$ and directed edge
$x_{i-1}x_i$ to the component $C_{t-1}$ thus extending the directed tree  formed by the occupied vertices and rooted the sink $z$.
 As the sink  is occupied from the start, any particle which reaches level $k$ automatically halts there.
The process ends at a  step $t_f$, {\em the finish time}, when the source vertex $v$ first becomes occupied by a halted particle. Thus $t_f$ is the final number of particles occupying the graph (not including the sink).

\paragraph{The models of this paper.}
Let $N_i=|S_i|$ be the size of the set $S_i$, the $i$--th layer of the graph $G$.
We regard all edges as directed from the source towards the sink.
We consider two models.
\begin{itemize}
\item The {\em  Cayley tree } $G(d,k)$ with branching factor $d$ and height $k$.
For convenience
we take the size of the last level to be $d^k=n$. Here $d\ge 2$ is fixed or tends to infinity with $n$
sufficiently slowly, so that $k=\log n/\log d$ also tends to infinity with $n$.
The source $v$ is the root vertex at level zero of the tree and the sink $z$ is
an artificial vertex connected to  all vertices in the final layer $S_k$ of the tree.
Excluding the sink,
 $G(d,k)$ is a $(d^{k+1}-1)/(d-1)$--vertex graph.

\item  The {\em multipartite layers model.} The sets $S_0, S_{k+1}$ have size one, the layers $S_i$ have size $N_i$. For $i=0,..,k$ there is a complete bipartite graph between $S_i$ and $S_{i+1}$. Here $k \ge 1$ is fixed or
    tends to infinity with $n$.

The {\em equal (multipartite) layers model} is a graph $M_E(n,k)$ in which
the sets of $S_i$  ($i =1,\ldots,k$) are the same size $N_i= n/k$,  where $N_i \rai$.
Excluding the source and sink,
 the equal layers model $G$ is an $n$--vertex graph. The parameter $k$ can
  either be a fixed integer $k \ge 1$, or a function of $n$. 
  In the extremes the graph has one layer, $S_1$, of $n$ vertices ($k=1$), or is a path ($k=n$).

The {\em growing layers model} is a graph $M_G(d,k)$ in which the sets $S_i$ grow geometrically in size
 with parameter $d$.
As  $|S_0|=1$, then $|S_i|=N_i=d^i$, and we take $d^k=n$. Here $d,k \rai$ with $n$. Excluding the sink,
  the growing layers model $G$ is a $(d^{k+1}-1)/(d-1)$--vertex graph.
\end{itemize}

Analysis of the DLA process is in terms of $n$, which is, up to a constant multiple,   the number of vertices in the graph. The  model is probabilistic and corrections arising from the exceptional events are estimated as a function of $n$, even if this is not always made explicit.
The parameter $k$  determines the number of levels in the graph, and $d$ the growth rate (if any) of the levels.
The value of $k$ or $d$  can be constant in some models or it can tend to infinity as a function of $n$ within some bounds. As $d^k=n$
in the growing layers model, they are related by $k=\log n/\log d$.

\paragraph{Notation.}
We say a sequence of events $\cE_n$, $n\geq 0$ occurs with high probability (w.h.p.) if $\Pr(\cE_n)=1-\e_n$ where $\e_n=o_n(1)$ for some function $o_n(1) \raz$ with $n$, and thus $\lim_{n\to\infty}\Pr(\cE_n)=1$.
 Typically we write  $o_n(1)=o(1)$ and similarly for other asymptotic notation such as $O(\cdot), \Th(\cdot), \Om(\cdot)$.
We use $A_n \sim B_n$ to denote $A_n=\ooi B_n$ and thus $\lim_{n \rai} A_n/B_n=1$.
We use $\om=\om_n$ in two ways; either to denote any quantity $\om_n$ which tends to infinity with $n$ but suitably slowly as required in the given proof context, or
as a fixed divergent quantity  whose  value is stated explicitly.
 The expression $ f(n) \ll g(n)$ indicates $f(n)/g(n)=o_n(1)$.
 The notation $f \rai$ indicates that $f=f(n)$ grows unboundedly with increasing $n$.
 All results claimed are for sufficiently large $n$.

\paragraph{Results for the multipartite layers model.}

At any step $t$, the occupied component $C_{t}$ is a tree  with edges directed downwards towards  the sink $z$.
At the finish time $t_f$  the source becomes occupied, and the component $C_{t_f}$ contains a directed   {\em connecting path}  from the source $v$ to the sink $z$, all of whose vertices contain halted particles.

On deletion of the sink vertex, the digraph $D_t=C_t \sm \{z\}$ consists of a directed forest  with components rooted at the occupied vertices of level $k$.
Let $v=u_0 u_1\cdots u_k u_{k+1}=z$ be  the path connecting source and sink at the end of the process.
In the multipartite layers model w.h.p.
  the tree component rooted at vertex $u_k$ containing the connecting path is precisely the  path $v=u_0 u_1\cdots u_k$. Thus this path grew back to the source without gaining any off-path neighbours due to other particles colliding with it.


\begin{theorem} \label{Th1}
Let $\om \rai$ slowly with $n$. 
The following results hold with high probability in the multipartite layers model.
\begin{enumerate}[(1)]
\item Equal layers model, $M_E(n,k)$. For $i=1,...,k$ let $N_i=n/k$.
Let
\[
T_f=  [ (k+1)!  (n/k)^k ]^{1/(k+1)} =  \g_k \,n^{k/(k+1)}
\]
where $\g_k>0$ constant, and $\g_k \ra 1/e$ as $k \rai$.

 Provided
 $2\le k \le \sqrt{\log n}/\log \log n$, the finish time $t_f$ satisfies
 $T_f/\om \le t_f \le \om T_f$.
 \item
Growing layers model, $M_G(d,k)$. For $i=0,...,k$, let $N_i=d^i$, where $d^k=n$ and thus $k = \log n/\log d$. Let
\[
T_f= \sqrt{k}\; d^{k+3/2-\sqrt{2k+2}}.
\]
\begin{enumerate}[(a)]
\item  Provided
 $d \rai,\; k \rai$ with $n$,  and  $k \le \log^2 d$,
  the finish time $t_f$ satisfies  $ T_f/\om \le t_f \le \om T_f$.


\item At $t_f$, the levels $i=1,\ldots, k-\rdup{\sqrt{2k+2}-1}$ contain a single occupied vertex, the vertex $u_i$ of  the connecting path.
\end{enumerate}

\item Let $v=u_0u_1\cdots u_k z$ be the path connecting source $v$ and sink $z$ at $t_f$,
and let $D_{t_f}=C_{t_f}\sm\{z\}$ be the occupied component with the sink deleted. In either model, with high probability, the vertices $\{u_1,\cdots ,u_k\}$ have in-degree one in $C_{t_f}$. Thus the tree rooted at $u_k$ in $D_{t_f}$ is exactly this path.
\end{enumerate}
\end{theorem}

\paragraph{Results for Cayley trees.}

For a Cayley tree $G=G(k,d)$, the connecting path is by definition the unique path from source to sink.
We give values for the finish time both for $d \rai$, and for $d$ finite.
By an indicative argument
Hastings and Halsey \cite{HH} derived an estimate of $\sqrt{2k}\,2^{k-\sqrt{2k}}$ for the expected finish time of DLA on the binary tree  $G(2,k)$ of height $k$.
We confirm their estimate is of the correct order of magnitude, and give  results  for $G(d,k)$ for any constant $d$.

\begin{theorem} \label{Th2}
With high probability the finish time $t_f$ of DLA on a  Cayley tree of  branching factor $d$ and height $k$
satisfies $T_f/\om \le t_f \le  \om T_f$, where $T_f$ is as given below, and where  $\om \rai$ slowly with $n$. 
\begin{enumerate}[(1)]
\item
If $d \rai,\; k \rai$, and $d \ll k$, then $T_f= \sqrt{k}\; d^{k+3/2-\sqrt{2k+2}}$.
\item
\label{d-tree}
If $d \ge 2$ constant, then $T_f=   \sqrt{k}\, d^{k-\sqrt{2k}}$.
\end{enumerate}
\end{theorem}

{
\paragraph{Comparison of occupancy.} How to compare the results of DLA on  different models with each other? One possibility is to define the {\em packing ratio} $\rho=N/|V|$ of a process as the number of particles $N$ at the finish divided by the number of vertices in the graph. Ignoring constants and lower order terms, $\r=n^{-1/(k+1)}$ for the equal layers model and $\r=n^{-\sqrt{2/k}}$
for the  growing layers model.

How can the packing ratio be used to compare results for DLA on infinite graphs such as $\bZ^d$? The sink in $\bZ^d$ is  the origin $\ul 0$, as this is where the initial particle is located.
At any fixed moment the longest arm of the figure formed by the occupied vertices defines a path back to the edge of the current bounding figure; which we take as the source (of particles entering from outside).

For the two-dimensional grid  it is a result of Kesten \cite{Kesten}, that,    when $N$ particles are added, the longest arm of the DLA figure is upper bounded by $C_2N^{2/3}$, for some constant $C_2$.
A circle of radius $N^{2/3}$ contains order $n=N^{4/3}$ vertices in $\bZ^2$, so
$N =\Th( n^{3/4})$.
Ignoring constants,  this gives $\r \le n^{-1/4}$  for $\bZ^2$, at any point in the (infinite) process.
For  $d \ge 3$, the longest arm length  in $\bZ^d$ is $C_d N^{2/d}$.  A figure of radius $R=N^{2/d}$ has order $n=R^d=N^2$ vertices, so $\r=n^{-1/2}$.
}

\paragraph{Road map of proofs.} The following is an outline  description and we often suppress minor details and qualifiers such as w.h.p.

The first step is to derive a solution to the recurrence  \eqref{Li}--\eqref{Lk} for the expected occupancy of the  levels in the multipartite layers model. Under certain assumptions, the solution is asymptotic to \eqref{muk}, and 
the actual occupancy is concentrated  around this value. This is particularly true in the equal layers model where
the  levels fill up in a more or less regular manner as given by \eqref{muk}, allowing us to estimate the finish time.

Things differ somewhat for the growing layers model. Although the higher levels fill in a regular manner, as given by \eqref{muk},
eventually we reach a level (with a well defined index $k-j^*$) where, in expectation, only a single occupancy occurs.  The expected waiting time for further occupancy of this level is much longer than the expected waiting time for a unique path of halted particles to grow backwards to the origin, thus terminating the process.
This is the content of Theorem \ref{Th1}.(2).(b). On the other hand Theorem \ref{Th1}.(3) (which requires a separate proof) says something  stronger. Namely that  the connecting path  grows back from level $k$ to the source as a unique path, and without gaining any off-path neighbours due to other particles colliding with it.

Because of the similarity in which the layers grow,
the results for the growing layers model tell us the likely behaviour of the Cayley tree model.
This allow us to construct  proofs for the finish time for Cayley trees. 

\section{Bounds on occupancy in the layers model}\label{Sec2}
Recall that $N_i$ is the size of layer $i$  for $i=0,1,...,k$, where the value of $N_i$ depends on the model in question.
For $t \ge 0$, let $L_i(t)$ be the number of  particles  halted in level $i$ at the end of step $t$. Thus $L_i(0)=0$ for all $i\le k$, and $t=\sum_{i=0}^k L_i(t)$. We refer to $L_i$ as the occupancy of level $i$.
Note that $L_0(t)=0$ for $t < t_f$, $L_k(t) \le t$, and generally $L_i(t) \le \min(t, N_i)$.

\paragraph{General formulation of layers occupancy.}

Let ${\cal H}(t)=(L_0(t),L_1(t),\ldots,L_k(t))$ be   the occupancy vector of the DLA process at step $t$. Then,
\begin{flalign}
\E (L_i(t+1) \mid \cH(t)) &=\; L_i(t) + \frac{L_{i+1}(t)}{N_{i+1}} \prod_{j=0}^i \brac{1-\frac{L_{j}(t)}{N_{j}}},\qquad i<k,
\label{Li}\\
\E (L_k(t+1) \mid \cH(t)) &=\; L_k(t) +  \prod_{j=0}^k \brac{1-\frac{L_{j}(t)}{N_{j}}}. \label{Lk}
\end{flalign}
Note that \eqref{Lk} follows from \eqref{Li} on defining $L_{k+1}(t)/N_{k+1}=1$ for all $t$, and  that if $L_0(t)=1$, the above recurrences give $L_i(t+1)=L_i(t)$.

The following proposition gives the solution to these recurrences under suitable conditions.
\begin{proposition}\label{Prop1}
For $t \ge  0$ let $\mu_k(t)=t$, and for $1 \le j \le k-1$, let
\begin{equation}\label{muk}
\mu_{k-j}(t)= \frac 1{N_k N_{k-1}\cdots N_{k-j+1}}\; \frac {t^{j+1}}{(j+1)!}.
\end{equation}

For $j \ge 0$, suppose there are steps $T_k <  \cdots<T_{k-\ell}<\cdots < T_{k-j}$, 
such that for all $\ell \le j$, at $T_{k-\ell}$    the value of $\mu_{k-\ell}(T_{k-\ell}) \rai$  sufficiently fast, and for all levels $i<k-\ell$ the value of $\mu_{i}(T_{k-\ell}) \raz$  sufficiently fast.
Then with high probability for all $\ell \le j$ and $T_{k-\ell}\le t \le t_f$ we have $L_{k-\ell}(t) \sim \mu_{k-\ell}(t)$ in   either of the multipartite layers  models.

The condition for the existence of $T_{k-j}$ is satisfied for $0 \le j \le k-1$ in the equal layers model  and for
$0 \le j \le \sqrt{2k+2}-2$ in the growing layers model.
\end{proposition}

The proposition describes a gap property, that when level $k-j$ starts to fill and become
concentrated, the levels with lower index $i=0,1,...,k-(j+1)$ remain empty.
The proof of Proposition \ref{Prop1} is inductive backwards from level $k$. 
The precise  growth of the levels, and the values of $j$ which make it work are to be determined.
The first steps  are common to the equal and growing layers models, and we include them together in this section.
For the equal layers model we complete the proof of Proposition \ref{Prop1} in Sections  \ref{StateL} and \ref{lbo}.
The analogous proof for the growing layers model is given   in Section \ref{Growing}.

\subsection{Upper bound on occupancy at step $t$}\label{secupper}

The underlying random walk $\r_t$ from  source $v$ to sink $z$ at step $t$ defines a path given by $v=u_0u_1\cdots u_ku_{k+1}=z$,  where $u_i$ is a random vertex in level $i$. Particle $t$ follows this walk until halting at a vertex $u_i$, where  $u_{i+1}$ is the first occupied vertex encountered by $\r_t$.

\paragraph{The upper-blocked process.}
Let $B_i(t)$ denote the occupied (blocked) vertices in level $i$ in the DLA process at the end of step $t$.
We  define  an {\em upper-blocked process} which we use to upper bound $L_i(t)$. This process gives rise to sets $\wh B_i(t)\supseteq B_i(t)$ and random variables $\whL_i(t)=|\wh B_i(t)|\geq L_i(t)$.
For {\em every} vertex $u_j$, $0 \le j \le k$ on the walk $\r_{t+1}$, if $u_{j+1}$ is occupied ($u_{j+1} \in \wh B_{j+1}(t)$) add a vertex to $\wh B_j(t)$ as follows. If $u_j \not \in \wh B_j(t)$ add $u_j$ to $\wh B_j(t+1)$. If $u_j  \in \wh B_j(t)$ add some other $u'_j \in S_j \sm \wh B_j(t)$ to $\wh B_j(t+1)$.
If a layer becomes full we continue with the layer above.
As we will prove, this contingency will not occur, as with high probability  the first layer to become full is the source.

If  particle $t+1$ halts at vertex $u_i$ in the DLA process, then either $u_i$ is added to  both  $B_i(t+1)$ and $\wh B_i(t+1)$, or $u_i$ is already a member of $\wh B_i(t)$. In either case $B_i(t) \subseteq \wh B_i(t)$ for all $i$ and $t \ge 0$.
It follows  that $L_i(t) \le \whL_i(t) \le t$, as $\r_{t+1}$ can add at most one vertex to $\wh B_i(t)$.    Moreover $\whL_k(t)=t$ deterministically (provided $t \le N_k$).
 As a consequence  the finish time $t_f(UB)$ of the upper blocked process satisfies $t_f(UB) \le t_f(DLA)$.

Let ${\wh{\cal H}}(t)=(\wh L_0(t),\wh L_1(t),\ldots,\wh L_k(t))$ be the occupancy vector of the upper-blocked process  at step $t$.
For $0 \le i \le k$, the expectation $\E\whL_i(t)$ satisfies the recurrence
\begin{flalign}
\E(\whL_i(t+1) \mid \wh{\cal H}(t))&=\; \whL_i(t) + \frac{\whL_{i+1}(t)}{N_{i+1}}\ind{\cE(t)},  \label{whLi}
\end{flalign}
where $\cE(t)$ is the event that $\whL_i(t)<N_i$ and that  $\whL_{i+1}(t) <N_{i+1}$ for $i<k$. As we only propose to analyse the process as long as no level is full, we assume $\ind{\cE(t)}=1$ forthwith.
Equation \eqref{whLi} follows because the upper blocked process increases the size of $\wh B_j(t)$ (if possible) whenever the walk $\r_t$  contains a vertex of $\wh B_{j+1}(t)$, this being true at all levels  $j=0,...,k$.

The evolution of ${\wh{\cal H}}(t)=(\wh L_0(t),\wh L_1(t),\ldots,\wh L_k(t))$ is Markovian, and for $t \le t_f$ we
henceforth assume for $i \ge 1$ that $\wh L_i(t) <N_i$ in our calculations.
at $t_f \le \om T_f$. If so, referring to \eqref{Li} and \eqref{Lk}, we have
\[
\E (L_i(t+1)\mid \cH(t))  \le L_i(t) +\frac{L_{i+1}(t)}{N_{i+1}}\le \whL_i(t) +\frac{\whL_{i+1}(t)}{N_{i+1}}=\E( \whL_{i+1}(t+1) \mid \wh\cH(t)).
\]

The next lemma gives w.h.p.  bounds for $\whL_i(t)$
when none of the layers $i=1,\ldots,k$ are full.
With high probability  the source is the only layer to become full in either the upper blocked and DLA process at or before $t_f$. When  $L_0(t)=1$ at $t=t_f$ the DLA process stops anyway.

The proofs of level occupancy are inductive backwards from level $k$. For a given level $k-j$ we identify two (not necessarily integer) times,
$t_1(k-j)$ and $t_{k-j}(\om)$, defined as follows.
For $j\ge 0$,  let $t_1(k-j)$ be the solution to $\mu_{k-j}(t)=1$.  
Thus as $\mu_k(t)=t$,  $t_1(k)=1$ and
\begin{equation}\label{t1-wun}
t_{1}(k-j)= [(j+1)! N_k N_{k-1}\cdots N_{k-j+1}]^{1/(j+1)}.
\end{equation}
Let $t_k(\om)=1$, and for $1\le j \le k-1$, let $t_{k-j}(\om)=(4\om^{3})^{1/(j+1)} t_1(k-j)$, so that
\begin{equation}\label{tom}
t_{k-j}(\om)= \brac{
(4\om^{3})[(j+1)! N_k N_{k-1}\cdots N_{k-j+1}]}^{1/(j+1)}.
\end{equation}
The variable $t_1(k-j)$ is used a reference point in many of our calculations, and  concentration of $\whL_{k-j}(t)$ follows for $t \ge t_{k-j}(\om)$.

\begin{lemma} \label{LHAT}  Let $\mu_{i}(t)$ be given by \eqref{muk}. Let  $\om=6 \log n$. Provided $\whL_i(t) \le N_i$, the following hold for $i=0,1,...,k$.
\begin{enumerate}[(1)]
\item   Deterministically $\whL_k(t)=t$, and for $j\ge 1$, if $j^2/t=o(1)$, then $\E\whL_{k-j}(t) \sim \mu_{k-j}(t)$.
\item  Suppose that $t_1(k-(j-1)) \ll t_1(k-j)$, and that $j^2/t=o(1)$.
If $t \ge t_{k-j}(\om)$  then  w.h.p. $\whL_{k-j}(t) =\mu_{k-j}(t)(1+ O(1/\om))$.
\end{enumerate}
\end{lemma}

{\bf Note.}
The fact that $t_1(k-(j-1)) \ll t_1(k-j)$ is to ensure that $\mu_{k-(j-1)}(t_1(k-j))$ is sufficiently large and thus $\whL_{k-(j-1)}(t_1(k-j))$ is concentrated as $\whL_{k-j}(t)$ grows, is a model dependent calculation given in Section \ref{StateL} and Section \ref{Growing} respectively for the equal and
  growing layers models.

\begin{proof}
As  $N_{k+1}=1$ and $\wh L_{k+1}(t)=1$, this implies that $\wh L_k(s)=s$ for $0\leq s\leq t$. Moreover at most one vertex can be added to $\whL_i(t-1)$ at step $t$, which implies $\whL_i(t) \le t$.

Iterating \eqref{whLi} backwards  for $0\leq s\leq t$, and using $\whL_i(0)=0$, gives
\begin{equation}\label{lambda}
\E\whL_i(t)= \frac 1{N_{i+1}}\; \sum_{s=0}^{t-1} \E \whL_{i+1}(s).
\end{equation}
We claim for $j \ge 0$ that
\begin{flalign}\label{whLk-j}
 \ind{t \ge j}\;\frac{(t-j)^{j+1}}{(j+1)!} \le (N_k N_{k-1}\cdots N_{k-j+1})\; \E \whL_{k-j}(t) \le   \frac{t^{j+1}}{(j+1)!}.
\end{flalign}
For given $t$, the induction is backwards on $k-j$ from $j=0$.
When $j=0$ \eqref{whLk-j} is true,
 so the first non-trivial case is $j=1$. From \eqref{lambda} we see that
\beq{6a}{
\E \whL_{k-1}(t)= \frac{1}{N_k} \sum_{s=0}^{t-1} s,
}
which illustrates how \eqref{whLk-j} arises from bounding this sum.

For the induction at step $i=k-(j+1)$,  let
\begin{equation}\label{M}
M_{j-1}=N_k N_{k-1}\cdots N_{k-j+1}.
\end{equation}
Multiply \eqref{lambda} by  $M_{j-1}$, and
insert the bounds on $M_{j-1} \E \whL_{k-j}(s)$ from \eqref{whLk-j} (with $i+1=k-j$) into this, to give
\begin{equation}\label{sumj+1}
\frac 1{N_{k-j}} \sum_{s=j}^{t-1} \ \frac{(s-j)^{j+1}}{(j+1)!}\le M_{j-1}\;\E \whL_{k-(j+1)}(t) \le  \frac 1{N_{k-j}} \sum_{{s=1}}^{t-1} \ \frac{s^{j+1}}{(j+1)!}\\
\end{equation}
By comparison of the sum with the related integral we have that
\begin{equation}\label{tum}
\frac{(t-1)^{m+1}}{m+1} \le 1^m+2^m+\cdots+(t-1)^m \le \frac{t^{m+1}}{m+1}.
\end{equation}
Use \eqref{tum} in \eqref{sumj+1} with $m=j+1$, 
giving
\[
\frac{\ind{t \ge j+1}}{N_{k-j}}\; \frac{(t-(j+1))^{j+2}}{(j+2)!}\le M_{j}\;\E \whL_{k-(j+1)}(t) \le  \frac 1{N_{k-j}}  \frac{t^{j+2}}{(j+2)!},\\
\]
 which completes the induction for \eqref{whLk-j}. Moreover,  provided $j^2/t=o(1)$,
\begin{equation}\label{ELhat}
\E \whL_{k-j}(t) =  \frac {1}{N_k N_{k-1}\cdots N_{k-j+1}} \frac{t^{j+1}}{(j+1)!}\brac{1 -O\brac{j^2/t}}
=\mu_{k-j}(t) \ooi.
\end{equation}
 This completes the proof of Lemma \ref{LHAT}.(1).
\end{proof}

We proceed to the proof of Lemma \ref{LHAT}.(2). The first thing  to check is that, for the values of $k$ given in Theorem \ref{Th1}, 
for $t \ge t_{k-j}(\om)$, $j^2/t=o(1)$,  allowing us to use Lemma \ref{LHAT}.(1).

\begin{lemma} \label{remark1}
Let $t_1(k-j)$ be the value of $t$ such that $\mu_{k-j}(t)=1$ as given by \eqref{t1-wun}.  The condition $j^2/t=o(1)$ is satisfied at $t_1(k-j)$ in the equal layers model provided $k=o(n^{1/5})$ and in the growing layers model provided  $d \rai$. This allows us to assume that for $t \ge t_{k-j}(\om)$, $\E \whL_{k-j}(t)\sim \mu_{k-j}(t)$ in subsequent calculations.
\end{lemma}
\begin{proof}
For $j \ge 1$, the value of $t_{k-j}(\om)$ given in \eqref{tom},
satisfies
\[
t_{k-j}(\om) \gg t_1(k-j)= ((j+1)!N_k\cdots N_{k-j+1})^{1/(j+1)}\ge M_{j-1}^{1/(j+1)},
\] see \eqref{M}.
The product $M_{j-1}$  is model specific,
having the values $M_{j-1}(E)=(n/k)^j$ (equal layers model) and $M_{j-1}(G)=d^{kj-j(j-1)/2}$ (growing layers model).

In the first case $M_{j-1}(E)^{1/(j+1)} \ge (n/k)^{1/2}$, and in the second $M_{j-1}(G)\ge d^{k/2}$, this minimum being achieved at $j=1$ or $j= k$. Checking $j^2/M_{j-1}^{1/(j+1)}$ we see that
 the condition $j^2/t=o(1)$ is satisfied at $t_1(k-j)$ in the equal layers model provided $k=o(n^{1/5})$ and in the growing layers model provided either $d \rai$ or $k \rai$.
\end{proof}


\subsection{Concentration of  $\whL_i(t)$ for sufficiently large $t$.}\label{HoConc}
 We now prove Lemma \ref{LHAT}.(2).
 \begin{lemma}\label{conc-whp}
Let $\mu_{k-j}(t)$ as be given by \eqref{muk}. Let $\om \ge 2 \log N_k+4 \log k$.  Let $t_k(\om)=1$, and for $1\le j \le k-1$, as given in \eqref{tom},
 let
\[
t_{k-j}(\om)= \brac{
(4\om^{3})[(j+1)! N_k N_{k-1}\cdots N_{k-j+1}]}^{1/(j+1)}.
\]
Let $k^*= \sqrt{\log n}/\log \log n$. If $k \le k^*$,
with probability $1-O(k^2\om N_k e^{-\om})$ it holds that  for  all $j \le k-1$ and all $t \ge t_{k-j}(\om)$  that $\whL_{k-j}(t) \in [\mu_{k-j}(t)(1-1/\om),\mu_{k-j}(t)(1+1/\om) ]$.
 \end{lemma}

 \begin{proof}
The proof is inductive and for level $k-j$ it depends on establishing the result for levels $k,k-1,\ldots,k-j+1$.
By definition, $\whL_k(t)=t=\mu_k(t)$. Let $t_k(\om)=1$, establishing Lemma \ref{LHAT}.3 for $j=0$.
For $j\ge 0$,  let $t_1(k-j)$ be the solution to $\mu_{k-j}(t)=1$ as given by \eqref{t1-wun}.

\ignore{
{\bf Note.}
The fact that $t_1(k-(j-1)) \ll t_1(k-j)$ ensuring that $\mu_{k-(j-1)}(t_1(k-j))$ is sufficiently large and thus $\whL_{k-(j-1)}(t_1(k-j))$ is concentrated, is a model dependent calculation given in Section \ref{StateL} and Section \ref{Growing} respectively for the equal and
  growing layers models.
The proof below assumes the truth of this.
}

For $j \ge 1$, 
\begin{equation}\label{CONC}
t_{k-j}(\om)=
(4\om^{3})^{1/(j+1)}[(j+1)! N_k N_{k-1}\cdots N_{k-j+1}]^{1/(j+1)}=(4\om^3)^{1/(j+1)}t_1(k-j),
\end{equation}
so that $\mu_{k-j}(t_{k-j}(\om))=4\om^3$ where $\om$ is still to be determined.

The random variable $\whL_{k-j}(t)$ is obtained from $\whL_{k-j}(t-1)$ by choosing a random  vertex $u \in S_{k-j}$  irrespective of the current occupancy of $u$, and  a random neighbour $w \in S_{k-j+1}$. If $w$ is occupied then $\whL_{k-j}(t+1)=\whL_{k-j}(t)+1$.
Thus, $\whL_{k-j}(t+1)=\whL_{k-j}(t)
+ Q_{k-j}(t)$, where $\Pr( Q_{k-j}(t)=1)=(\whL_{k-j+1}(t)/N_{k-j+1})$ independently of any previous outcomes.
In particular, $\E Q_{k-1}(t)=t/N_k$, and  by  equation \eqref{6a}, $\E \whL_{k-1}(t)=t(t-1)/(2N_k)\sim \mu_{k-1}(t)$.
Similarly $\E \whL_{k-j}(t)\sim \mu_{k-j}(t)$ again by the Proof of Lemma \ref{LHAT}.2.

For $j \ge 1$, and given $ t \ge t_{k-j}(\om)$ let $\cA_j(t)$
denote the event that $\wh L_{k-j}(t)\in [\m_{k-j}(t)(1- 1/\om), \m_{k-j}(t)(1+ 1/\om)]$. By Hoeffding's Inequality,
\begin{equation} \label{nEjt}
\Pr( \neg \cA_j(t))
\leq 2\exp\set{-\frac{\m_{k-j}(t)}{3\om^2}\brac{1-\d}}
\le2\exp\set{-\frac{\mu(t_{k-j}(\om))}{4\om^2}} \le 2e^{-\om},
\end{equation}
where  $\d=O(1/{\om}+j^2/t)$, includes the correction from \eqref{ELhat}.

Let $\cE_j(t)$ be the event that $\cA_j(s)$ holds for all $s \in [t_{k-j}(\om), t]$. For $j=0$, $\Pr(\cE_{0}(t))=1$, and for $j\ge 1$, given $\cE_{j-1}(t)$,
we have inductively that
\begin{equation} \label{sumup}
\Pr(\neg \cE_{j}(t)\mid \cE_{j-1}(t)) \le \sum_{s=t_{k-j}(\om)}^t \Pr(\neg \cA_j(s)).
\end{equation}
Adding over $i \le j$ we will be able to complete the induction,  via
\begin{equation}\label{sumupj}
\Pr(\neg\cE_j)\leq \Pr(\neg\cE_j\mid \cE_{j-1})+\Pr(\neg\cE_{j-1})\le\sum_{i\le j}  \sum_{t  \ge t_{k-i}(\om)}  \Pr(\neg\cA_i(t))\le c j \om N_k e^{-\om},
\end{equation}
provided we can establish the bound on the RHS, which we now do.
Here $c$ is some absolute constant.

As $\mu_{k-j}(t)$ in \eqref{muk} is monotone increasing in $t$ and by \eqref{nEjt} is bounded above by $2e^{-\om}=o(1)$,  we 
use the Euler-MacLaurin Theorem to
replace the summation over $t$ in \eqref{sumup} by an integral.
Thus,
\[
\sum_{t  \ge t_{k-j}(\om)}  \Pr(\neg\cA_j(t)) \le 2\sum_{t \ge t_{k-j}(\om)}\exp\set{-\frac{\mu(t_{k-j}(\om))}{4\om^2}} \le 3 \int_{t \ge t_0}
\exp\set{-\frac{t^{j+1}}{C}}\, dt,
\]
where, $t_0=t_{k-j}(\om)$,  $C=[4\om^2 (j+1)! N_k \cdots N_{k-j+1}]$, and from \eqref{CONC}, $t_{k-j}(\om)=(4\om^3 C)^{1/j+1}$.

Put $t^{j+1}/C=z^2/2$, so that $t=(Cz^2/2)^{1/(j+1)}$, and $dt/dz= (2/(j+1))(C/2)^{1/(j+1)}z^{2/(j+1)-1}$.
Let $z_0=z(t_0)$, then $z_0=\sqrt{2 \om}$, and as $j\ge 1$,
$z^{2/(j+1)-1}\le 1$. Finally  $(C/2)^{1/(j+1)} = O(\om^{3/2} N_k)$, giving
\[
\int_{t \ge t_0}
\exp\set{-\frac{t^{j+1}}{C}}\, dt \le \frac{2}{j+1}\bfrac{C}{2}^{1/(j+1)} \int_{z \ge z_0} e^{-z^2/2}\,dz
\le O(\om^{3/2} N_k)\; \frac{1}{\sqrt \om} e^{-\om} ,
\]
by using a standard bound on the tail of the Normal
distribution, $\Pr(Z \ge z) \le 1/(\sqrt{2\pi}z) e^{-z^2/2}$.

Lemma \ref{conc-whp}  follows for all $1\le j \le k$ by adding the RHS of \eqref{sumupj} over $j \le k$ and choosing
$\om=2 \log N_k+ 4 \log k \le 6 \log n$
\end{proof}

\ignore{
Choosing $\om=6 \log n$ satisfies Lemma \ref{LHAT}, Lemma \ref{conc-whp},
and the next lemma, which we need for the lower bound on the DLA occupancy.

\begin{lemma}\label{NotBig}
Let $\om=6\log n$, $1 \le j \le k-1$, and $t_{k-j}(\om)$ be given by \eqref{tom}. Let $\cE$ be the event that for some $j$ and some $t \le t_{k-j}(\om)$ it holds that
$\whL_{k-j}(t) \ge \om^4$. Then $\Pr(\cE)=o(1/n^2)$.
\end{lemma}
\begin{proof}
This follows by an application of a Chernoff bound. The value of $\mu_{k-j}(t)$ is monotone increasing in $t$
and takes value $ 4\om^3$ at $t^*=t_{k-j}(\om)$. For any $c>0$ constant
\[
\Pr(\whL_{k-j}(t^*) \ge \om^4) \le (4e/\om)^{\om}=O(n^{-c}).
\]
Thus choosing $c>4$, $\Pr(\cE)\le n^2 O(n^{-c})=o(1/n^2)$.
\end{proof}
}

\subsection{Lower bound on occupancy at step $t$}
So far we only have an upper bound on $\E L_j(t)$ given by $\E L_j(t) \le \E \whL_j(t)$. We construct a lower bound
$\E \wtL_j(t)$ and prove that for large enough $t$  these bounds converge thus giving the asymptotic value of $\E L_j(t)$.

For a given level $j$, a  lower bound  on $L_j(t)$ can be found as follows. Define a sub-process of DLA
which requires that the particle avoids  upper-blocked vertices.
Thus to reach level $j$, a particle must avoid choosing neighbours in $\wh B_1,...,\wh B_j$ at steps $0,...,j-1$ of its random walk.

Let $\stL_j(t)$ be a w.h.p. upper bound on $\whL_j(t)$. This will, for example, be obtained from Lemma \ref{conc-whp}. if $t \ge t_j(\om)$.
Referring to \eqref{Li}, \eqref{Lk},
let $\wtL_i(t)$ be obtained by replacing $L$ by $L^*$ in  the bracketed terms on the RHS.
This defines a lower bound  $\wtL_j(t)$, such that w.h.p. $\wtL_j(t)\le  L_j(t) \le \whL_j(t)$.
Let $\wt {\cal H}(t)=(\wtL_0(t),\wtL_1(t),\ldots,\wtL_k(t))$ be   the occupancy vector obtained from this lower bound,
then we have the following recurrences.
\begin{flalign}
\E (\wtL_i(t+1)\mid  \wt {\cal H}(t)) &=\;  \wtL_i(t) + \frac{ \wtL_{i+1}(t)}{N_{i+1}} \prod_{j=0}^i \brac{1-\frac{\stL_{j}(t)}{N_{j}}},\label{LLi}\\
\E (\wtL_k(t+1)\mid \wt {\cal H}(t))  &=\; \wtL_k(t) +  \prod_{j=0}^k \brac{1-\frac{\stL_{j}(t)}{N_{j}}}.\label{LLk}
\end{flalign}
These recurrences are solved in  Section \ref{lbo} for the equal layers model, and in Section \ref{Growing} for the growing layers model.

\paragraph{The gap property.} It is clear that $0 \le \wtL_{k-j}(t) \le L_{k-j}(t) \le \whL_{k-j}(t) \le t$, and that these values are monotone non-decreasing.
We choose times
\[
t^-(k-j) < t_1(k-j) <t^+(k-j)=t_{k-j}(\om)
\]
such that
\[
\mu_{k-j}(t^-) =1/\om^3, \quad \mu_{k-j}(t_1)=1, \quad \mu_{k-j}(t^+)=4\om^3.
\]
Below $t^-(k-j)$, $\whL_{k-j}(t)=0$ w.h.p. for all $j=1,...,k$. This follows from the definition of $k< \om$ and the Markov Inequality. Above $t^+(k-j)$ we have  that $\E \wt L_{k-j}(t) \sim  \E\whL_{k-j}(t)$ and thus $\E L_{k-j}(t) \sim \mu_{k-j}(t)$.  The main content of  Sections \ref{EqualL} and \ref{Growing} is to prove this via a {\em gap property} which implies that
\[
t^+(k-j)= t_{k-j}(\om)\; \ll \;t^-(k-(j+1))\; \ll \; t_1(k-(j+1)),
\]
so that w.h.p. concentration for  $L_{k-j}$ occurs while $L_{k-(j+1)}$ is still zero.

\section{Analysis of DLA in the equal layers model}\label{EqualL}

It is convenient at this point to obtain an asymptotic estimate for the finish time of DLA in the equal layers model.
Let $t_1(k-j)$ and $t_{k-j}(\om)$ be as given by \eqref{t1-wun} and \eqref{tom}. Thus $T_f=t_1(0)$, the time at which the source has expected occupancy one
(in the upper-blocked process).
Recalling that  $N_1=\cdots=N_k=n/k$, let
\begin{equation}\label{Tef}
T_f= \left [ (k+1)!  N_1\cdots N_k \right ]^{1/(k+1)}=\brac{\frac{(k+1)!}{k^k} n^k}^{1/(k+1)}.
\end{equation}
 That $T_f$ indeed approximates the finish time
will be shown in Section \ref{fini}.
\begin{lemma}\label{Tfval}
For $k \ge 2$, let $G$ be an equal layers graph with level sizes $N_i= n/k$, $i=1,...,k$.
\begin{enumerate}[(1)]
\item
Let $\g_k=e^{-1} (1+O(\log k/k))$, then
\begin{equation}\label{Tgk}
T_f=n^{k/(k+1)}\g_k.
\end{equation}
\item For $1\le j \le k$, let  be as given in \eqref{t1-wun}, then
\[
t_1(k-j)=\brac{(j+1)!}^{1/(j+1)}(n/k)^{j/(j+1)}= \Th(j)(n/k)^{j/(j+1)},
\]
and
\[
t_1(k-j)=t_1(k-(j-1)) \cdot \Th(1) (n/k)^{1/(j(j+1))}.
 \]
Let $k^*= \sqrt{\log n}/\log \log n$. If $k \le k^*$, then $t_1(k-j) \gg t_1(k-(j-1))$ in the equal layers model; as claimed in the note below Lemma \ref{LHAT}.
\end{enumerate}
\end{lemma}
\begin{proof}
 Note that
\[
\frac{(k+1)!}{k^k}=e^{O(1/k)}\frac{1}{k^k}\sqrt{2\pi}e^{-(k+1)}k^{k+3/2}(1+1/k)^{k+3/2}=e^{-(k+1)}\;\Th_k,
\]
where $\Th_k\sim(e^{1-1/k+O(1/k^2)}\sqrt{2\pi} k^{3/2})$. Thus
\[
(\Th_k)^{1/k+1}=e^{O( \log k/k)}=1+O(\log k/k)
\]
 is  bounded for $k \ge 2$ and tends to one  as $k \rai$.
From \eqref{Tef}
\begin{equation}\label{whf}
T_f=\brac{ \frac{(k+1)!}{k^k}{n}^{k}}^{1/(k+1)}
=n^{\frac k{k+1}}e^{-1}\; (1+O(\log k/k)).
\end{equation}
The second part follows by direct calculation using $\brac{(j+1)!}^{1/(j+1)}=j\Th(1+j^{3/(2j)})=j\Th(1)$.
\end{proof}
Note that with
$t_{k-j}(\om)$ as given by \eqref{tom}, then $t_{k-j}(\om)=(4\om)^{1/(j+1)}t_1(k-j)$,
so  asymptotics follow from the above lemma.

\subsection{Evolution of the state vector $\whL$ in the equal layers model.}\label{StateL}
We prove there  is a large gap in the number of steps between  the time when $\E \whL_{k-j}=1$ with all lower values zero, and the time when $\E \whL_{k-(j+1)}=1$. The gap  allows $\whL_{k-j}$ to increase and  become concentrated around $\mu_{k-j}$, whilst  all  values with a lower index $i<k-j$ remain zero. This
 confirms the inductive assumption stated below \eqref{t1-wun} in Lemma \ref{conc-whp}.

We list various  assumptions used in this section.
\begin{equation}\label{ometc}
 1 \le k \le k^*=\frac{\sqrt{\log n}}{\log \log n},\quad  \om  =6\log n, \quad \b= \frac{N_k}{\om T_f}
 \ge (6\log n)^{4+k/2}.
\end{equation}
Let $\whL=(\whL_0, \whL_1,\ldots, \whL_k)$ be the state vector of the upper-blocked  process. The entries in $\whL$ are non-negative integers, and if $\whL_i=0$, then $\whL_{i-1}=0$.

 The following  argument for $t \le \om T_f$ proves there is a large enough gap $t''-t$ between $\mu_{k-j}(t)=1$ and $\mu_{k-(j+1)}(t'')=1$ for $\whL_{k-j}(t'')$ to be concentrated, as assumed in Lemma \ref{conc-whp}. In particular w.h.p. at
 $t_{k-j}(\om)$, where $t''\gg t_{k-j}(\om) > t'$ we have \eqref{L2}.

Define $\b=\b_k=N_k/\om T_f$.
From \eqref{Tgk},
\begin{equation}\label{beta}
\b=\frac{N_k}{\om T_f}=\frac{n}{ \om k}\frac 1{\g_k n^{k/(k+1)}}=\frac{1}{\g_k} \frac 1{\om k}n^{1/(k+1)}.
\end{equation}
We assumed that $\b \ge (6\log n)^{4+k/2}$.
This is true if  $\om= 6 \log n$ since we assume that $k \leq \sqrt{\log n}/\log \log n$.

As $N_{k-j}=N_k=n/k$ in the equal layers model, for any $t \le \om T_f$,
\begin{align}\label{dogs}
\mu_{k-(j+1)}(t) &= \frac{t}{(j+2){ N_{k-j}}}\cdot \mu_{k-j}(t) \le \frac 1{(j+1)\b}\mu_{k-j}(t).
\end{align}
For $j+\ell \le k$, we can iterate this to give
\begin{align} \label{doggy}
\mu_{k-(j+\ell)}(t) &\le \mu_{k-j}(t)\frac 1{\b^\ell} \frac{1}{(j+\ell+1)(j+\ell)\cdots (j+2)}\le \mu_{k-j}(t)\frac 1{((j+1)\b )^\ell}.
\end{align}

Consider $\E\whL_{k-j}(t)$. At $t \sim t_1(k-j)$ when $\mu_{k-j}(t)\sim 1$, then   (see Lemma \ref{remark1}) $\E \whL_{k-j}(t) \sim 1$.
The Markov inequality implies that w.h.p. $\whL_{k-j}(t) \in I_\om=[0,1,...,\om]$. From \eqref{dogs}--\eqref{doggy},
\begin{equation}\label{dogs1}
 \mu_{k-(j+\ell)}(t) \le \frac {1+o(1)}{((j+1)\b)^\ell}, \qquad\text{for } 1\le \ell\leq k-j,
\end{equation}
and thus w.h.p. $\whL_0=0,\whL_1=0,\ldots,\whL_{k-(j+1)}=0$.
Using \eqref{dogs}--\eqref{doggy}, we see that
\begin{equation}\label{dogalog}
 \mu_{k-j+\ell}(t) \ge \mu_{k-j}(t) \; j(j-1)\cdots(j-\ell+1) \,\b^{\ell}\qquad \text{for } \ell\leq j.
\end{equation}
So if $\mu_{k-j}(t)\sim 1$, then  for $1\le \ell \leq j$,  $t\ge t_{k-j+\ell}(\om)$ and
\begin{equation}\label{6logn}
\mu_{k-j+\ell}(t) \ge \b^\ell \ge (6\log n)^{4+k/2}.
\end{equation}
Thus Lemma \ref{conc-whp} holds,
and  w.h.p. $\whL_{k-j+\ell}$ is equal to $(1+o(1))\E \wh L_{k-j+\ell}\sim\mu_{k-j+\ell}(t)$.

In summary, at  time $t$ such that $\mu_{k-j}(t) \sim 1$  (implying that $t\leq T_f$, see Proposition \ref{Tfval}.(2)), w.h.p., the state vector  $\whL$ is such that $\mu_{k-\ell}\rai$ for $\ell \le j-1$, and
\begin{equation} \label{L1}
\whL(t)=(0,\ldots,0,\;\whL_{k-j}\in I_\om,\; (1+o(1))\mu_{k-j+1},\; (1+o(1))\mu_{k-j+2},\ldots,\;\mu_k).
\end{equation}
Let $t=t_1(k-j)$,  let $t' =
t_{k-j}(\om)=(4\om^3)^{1/((j+1)} t_1(k-j)$, so $\mu_{k-j}(t') \sim 4\om^3$.  By  Lemma \ref{conc-whp} it holds w.h.p. that $\whL_{k-j}(t')\sim \mu_{k-j}(t')$.

Let $t''=t_1(k-(j+1))$. By Proposition \ref{Tfval}.(2),\; $t_1(k-(j+1))=\Th(1) (n/k)^{1/(j(j+1))}  t_1(k-j)$.
As $t'=(4\om^3)^{1/((j+1)}t_1(k-j)$, it can be checked that $t'' \gg t'$.
Also by \eqref{doggy}
\[
\sum_{\ell \ge 1} \mu_{k-(j+\ell)}(t')=O\bfrac{\mu_{k-j}(t')}{\b}=O\bfrac{\om^3}{\beta}=O\bfrac{1}{\log^{1+k/2}}.
\]
Thus, applying the Markov inequality to the above,  at $t'=t_{k-j}(\om)$, w.h.p.
\begin{equation} \label{L2}
\whL(t_{k-j}(\om))=(0,...,0,\;(1+o(1))\mu_{k-j},\;  (1+o(1))\mu_{k-j+1},...,\;\mu_k),
\end{equation}
so that $\whL_{k-j}(t')$ is concentrated and all lower levels are unoccupied, and thus  the claimed gap exists.
This condition persists  w.h.p. until around $t_1(k-(j+1))=t''\gg t'=t_{k-j}(\om)$, when $\mu_{k-(j+1)}(t_1(k-(j+1)))=1$
at which point $\whL(t'')$ resembles \eqref{L1} and the induction continues.

\subsection{Lower bound  on occupancy in the equal layers model.}\label{lbo}
We prove that, for $t \ge t_i(\om)$, we have $\E \wtL_i(t) \sim \E \whL_i(t) \sim \mu_i(t)$.
As with the upper bounds, we will have that  $\E \wtL_k \gg \E \wtL_{k-j}$ for $j \ge 1$. The first step is to
draw a line between them.

Let $t^-=t_1(k-1)/\om$, so that $\mu_{k-1}(t^-)=1/\om^2$. Then w.h.p $\whL_{k-1}(t^-)=0$ and as $\whL_{k-1}(t)$ is monotone non-decreasing,  w.h.p. $\whL_j(t)=0$ for all $j \le k-1$ and $t \le t^-$. As $\whL_k(t)=\stL_k(t)=t$ deterministically, this simplifies \eqref{LLk} for $t \le t^-$.

As before let $\b=N_k/\om T_f$ where $\om,k$ are given by \eqref{ometc}.
Provided $k \ge 2$, if $t \ge t^-$, then $t/\b \gg \om^3$. Indeed using \eqref{beta},
\[
t^-=\frac{t_1(k-1)}{\om} = \frac{1}{\om} \sqrt{\frac{2n}{k}} \qquad \implies \qquad \frac{t^-}{\b} = \frac{\g_k \om k}{n^{1/(k+1)}} \cdot \frac{1}{\om} \sqrt{\frac{2n}{k}} \ge 2\g_k n^{1/6} \gg \om^3.
\]

For $t \ge t^-$, and  $1\le i \le k-1$ we first  consider the product term in \eqref{LLi}.
Recalling that $N_i=N_k=n/k$,  we will prove that
\begin{equation}\label{small}
\prod_{j=0}^i \brac{1-\frac{\stL_{j}(t)}{N_{j}}} \ge 1- \sum_{j=0}^i \brac{\frac{\stL_{j}(t)}{N_{j}}}
=1-O\bfrac{t}{\b N_k}=1-o(1).
\end{equation}

By \eqref{doggy}, $\mu_{k-j} \le \mu_k/\b^j$.
Assume $t \ge t^-$ and that  for some $\ell \le k-1$, $t_{\ell}(\om) \le t < t_{\ell+1}(\om)$. Apply Lemma \ref{conc-whp} for $j\ge \ell+1$ 
along with \eqref{L1} and \eqref{L2}.

 If $i< \ell$ then $\stL_{j}(t)=0$, for $j \le i$.
 Next, if $i=\ell$, $L^*_j(t)=0$ for $j<i$ and $L^*_i(t)\le 2 \mu_i(t_i(\om))\le 8\om^3$. Finally  assume $i \ge \ell+1$. Then (as $t \ge t^-$),
\begin{equation}\label{bound}
\sum_{j=0}^i \frac{\stL_{j}(t)}{N_{j}}\le \frac{8\om^3}{N_k}+2\sum_{j=\ell+1}^{i} \frac{t}{\b^{k-j} N_k} \le \frac{O(1)}{N_k}\brac{\om^3+\frac{t}{\b}}=O\bfrac{t}{\b N_k}.
\end{equation}

Consider now $\E \wtL_k(t)$. For all $t \ge 0$,  $\whL_k(t)=t$, and $\sum_{j<k} \whL_j(t)=O(t/\b)$.
So from  \eqref{LLk}
\[
\E \wtL_k(t)=t -\sum_{ s=0}^t O\bfrac{s}{N_k}=t-O\bfrac{t^2}{N_k}=t\brac{1-O\bfrac{t}{N_k}}.
\]
For $t\ge t_0$ where $t_0 \rai$
arbitrarily slowly we have $\wtL_k(t) \sim t$ w.h.p., initializing an induction for $\E \wtL_i(t)$
using arguments equivalent  Section \ref{HoConc} for $\E\whL_i(t)$.

At step $t+1$, equation \eqref{small} implies that,  in the lower bounds on the process, particle  $t+1$ arrives at level $i$  with probability $(1-o(1))$. If $i<k$, it halts at this level with probability $\wtL_{i+1}(t)/N_{i+1}$.  Thus
\begin{equation}
\E\wtL_i(t+1) =\; \E\wtL_i(t) + \brac{1-O\bfrac{t}{\b N_k}}\cdot\frac{\E \wtL_{i+1}(t)}{N_{i+1}}. \label{wtLi}
\end{equation}
Arguing as in \eqref{lambda} on the inductive
assumption that
$\E \wtL_{i+1}(t)=\mu_{i+1}(t) (1-O(t/N_k))$, we find
\begin{flalign*}
 \E \wtL_i(t)=&\frac{1}{N_{i+1}} \sum_{s=0}^{t-1} \E \wtL_{i+1}(s) \brac{1-O\bfrac{s}{\b N_k}}\\
 =&\frac{1}{N_{i+1}} \sum_{s=0}^{t-1} \mu_{i+1}(s)\brac{1-O\bfrac{s}{N_k}} \brac{1-O\bfrac{s}{\b N_k}}\\
 =& \frac{1}{N_{i+1}} \sum_{s=0}^{t-1} \frac{s^i}{i! N_k \cdots N_{i+2}} \brac{1-O\bfrac{s}{N_k}}\\
=&\mu_i(t)-O(1)\frac{t \mu_{i}(t)}{N_k}\;\;
=\;\mu_i(t)\brac{1- O\bfrac{t}{N_k}}.
\end{flalign*}
Thus
\begin{equation}\label{Lit}
\E \wtL_i(t) \sim \E L_i(t) \sim \E  \whL_i(t)\sim \mu_i(t)= \frac{t^{i+1}}{(i+1)! N_k \cdots N_{i+1}}
\end{equation}
as required.

Let $t_i(\om)$ be given by \eqref{CONC}.
For those $i \le k$, such that $t \ge t_i(\om)$, then $\mu_i(t) \rai$  suitably fast
and the concentration results of Lemma \ref{conc-whp} hold. The gaps inherited from the upper bound argument of Section \ref{StateL} are essentially unaltered.

This completes the proof of Proposition \ref{Prop1} for the equal layers model.

\subsection{Finish time of DLA in the equal layers model.}\label{fini}
A lower bound on the finish time follows from the upper-blocked process, and an upper bound from the lower bound estimates for DLA.

\begin{proposition}\label{endDLA}
For $ 1 \le k \le k^*=\sqrt{\log n}/\log\log n$, let $G$ be an equal layers graph with level sizes $N_i= n/k$, $i=1,...,k$.  Let  $T_f$ be given by \eqref{Tef}.
With probability $1-O(1/\om')$, the finish time $t_f$ of the DLA process in $G$ satisfies $T_f/\om' \le t_f \le \om' T_f$, where $\om' \rai$ arbitrarily slowly.
\end{proposition}
\begin{proof}
Let $t_1=t_1(0)$ be such that $\mu_0(t_1)\sim 1$, and thus $t_1\sim T_f$. Let $t'=t_1/\om'$, where $\om' \rai$ slowly. Then w.h.p.,
$\E \whL_0(t')=O(1/(\om')^{k+1})$, and thus $\Pr(\whL_0(t')>0)=O(1/\om')$.

We next investigate the concentration of $L_1(T_f)$. By Lemma \ref{Tfval}.(1),
\begin{equation}\label{atlast}
 \mu_1(T_f)=\mu_0(T_f) \frac{N_1 (k+2)}{T_f}\ge \frac{n}{\g_k n^{k/(k+1)}}=\Th(1) n^{1/(k+1)} \gg 4 \om^3,
\end{equation}
where $\om=6 \log n$. Thus $T_f \gg t_1(\om)$ and hence  Proposition \ref{conc-whp} holds for $\wtL_1(T_f)$.
Suppose  at $T_f$ that  $L_0(T_f)=0$. By \eqref{atlast} the expected waiting time $\t$ for a particle to hit $L_1(T_f)$ is
\[
\t\le (1-o(1)) \frac{N_1}{\mu_1(T_f)}=\Th(1)  \frac n{k\, n^{1/(k+1)}}=\Th(1) \frac{n^{k/(k+1)}}{k}=\Th(T_f/k)
=O(T_f)
\]
By  time $\om' T_f$, w.h.p. $L_0(\om'T_f)=1$,  completing the proof of Proposition \ref{endDLA}, and hence Theorem \ref{Th1}.(1).
\end{proof}

\subsection{Existence of a unique connecting path component.}\label{uniqp}
We now prove Theorem \ref{Th1}.(3) for the equal layers model. We must show w.h.p. that at $t_f$, the finish time, the path of
occupied vertices connecting the source to level $k$ (and hence to the sink) has no off-path neighbours in $C_{t_f}$.

At a given step $t$, the edge induced component $C_t$ is obtained from $C_{t-1}$ by adding a newly occupied vertex which points to the neighbour in $C_{t-1}$ which halted the particle:
Thus a particle  halts at  vertex $u$ in level $i$ if  it  chooses an edge $uw$ to an occupied neighbour $w$ in level $i+1$. We consider this edge $uw$ as being directed from $u$ to $w$ in the component $C_t$ rooted at the sink.
An arborescence is a rooted  tree with all edges directed towards the root vertex.
Thus $C_t$ is an arborescence with root $z$. On deletion of the  $z$, $D_t=C_t\sm\{z\}$ becomes a directed forest of arborescences each rooted at a vertex in level $k$.

Let $B_i$  be the set of occupied vertices in level $i$, where $L_i=|B_i|$. Given that a particle at  $u$ chooses a vertex  in $B_{i+1}$ as a neighbour, then this neighbour is chosen uniformly at random (u.a.r.)
from the set $B_{i+1}$.

We regard  vertices occupied by halted particles as coloured either red or blue, with all occupied vertices in level $k$ coloured blue.
If  $u$ is the first in-neighbour of $w$ then $u$ is coloured blue. If however $w$ already has an in-neighbour $u'$, then $u, u'$ and all other in-neighbours are (re-)coloured red. At any step, the red vertices in a level
are those with siblings, and the blue ones are the unique in-neighbour of some vertex in the next level.
The choice of $w$ by the particle at $u$ is independent of the colour of $w$ at this step.

The process halts  when there is a directed path of occupied vertices $v=u_0 u_1\cdots u_kz$ from the source to sink.
 The source vertex $v$ is blue at $t_f$ as it is the first in-neighbour of $u_1$.
\begin{lemma}\label{blue}
With high probability, the path $v=u_0 u_1\cdots u_k =w$ from the source to level $k$ is blue, and thus the arborescence of halted particles rooted at $w=u_k$ is exactly this path.
\end{lemma}
\begin{proof}
As before, let $B_i$  be the set of occupied vertices in level $i$, and $L_i=|B_i|$. As each $u \in  B_i$ has a unique occupied out-neighbour in $C$, the subset {\rm Out}$(B_i)$ of $B_{i+1}$ with at least one in-neighbour has size at most $L_i$.

Let $\ind{k-j,s}$ be the indicator that particle $s$ halts in level $k-j$ and is coloured red due to a pre-existing sibling. In this case particle $s$ has chosen an out-neighbour in the existing set {\rm Out}$(B_{k-j}) \seq B_{k-j+1}$, and thus, as $|${\rm Out}$(B_{k-j})| \le \whL_{k-j}(s)$,
\[
\E \ind{k-j,s} \le \frac{\E \whL_{k-j}(s)}{(n/k)} \sim \frac{\mu_{k-j}(s)}{(n/k)}.
\]
Let $Z_{k-j}(t)$ be the number of red vertices in level $k-j$. We  associate the number of possibly-red
vertices at each step with a super-process $\wh Z_{k-j}(t)=Z_{k-j}(t)+Q_{k-j}(t)$, where  
$Q_{k-j}(t)$ is Bernoulii with parameter $(\whL_{k-j}-|(${\rm Out}$)B_{k-j}|)$.
Thus
$Z_{k-j}(t) \le \wh Z_{k-j}(t)$ and where
\[
\E(\wh Z_{k-j}(t+1) \mid \wh{\cal H}(t))=\wh Z_{k-j}(t)+\frac{\whL_{k-j}(t)}{N_{k-j}}.
\]
The recurrence  for $\E \wh Z$ mirrors the recurrence \eqref{whLi} for $\whL_{i}(t)$ with $k-j$ replacing $i+1$. Specifically,
\begin{align}
\E \wh Z_{k-j}(t) \sim & \;\frac 1{(n/k)}\sum_{s=1}^t{\mu_{k-j}(s)}= \frac{1}{(n/k)} \sum_{s=1}^t \frac{s^{j+1}}{(j+1)! (n/k)^j} \nonumber\\
 \sim &\;  \frac{t^{j+2}}{(j+2)! (n/k)^{j+1}}= \mu_{k-(j+1)}(t). \label{ZZZ}
\end{align}
We note that $\wh Z_{k-j}(t)$ is the sum of independent indicator variables, and will be concentrated at or after
$t_{k-(j+1)}(\om)$.
The number of red vertices is at most $2Z_{k-j}(t) \le 2 \wh Z_{k-j}(t)$, where
the factor 2 covers the case where the pre-existing  sibling was blue but is  recoloured red.

Assume $T_f/\om \le t_f \le \om  T_f$, where $\om=\om'$ from Proposition \ref{endDLA} tends slowly to infinity with $n$. Denote the path connecting $v$ to level $k$  by $v=u_0 u_1 u_2\cdots u_k$. By definition $v=u_0$ is blue.
For $i \ge 1$, let $R_i(s)$ be the red vertices in level $i$ at step $s$. Let $s_i$  denote the step at which $u_i$ became occupied.
Consider next the colour of $u_1$ at $t_f$.
 As  $T_f= t_1(0)$, and $t=t_f \ge T_f/\om\gg t_1(\om)$ by the gap property, so we have that $L_1(t) \sim \mu_1(t)$. On the other hand $\wh Z_1(t)$ may not be concentrated, but we can assume $\wh Z_1(t) \le \om' \mu_0(t)$ where $\om'$ is to be determined. If $t=2\om T_f$, then $\mu_0(t)=(2\om)^{k+1}$ and as $k\ge 2$ then $t \ge t_0(\om)$, so the red subprocess in level one is concentrated. Crudely\footnote{Why? The number of red vertices can only increase with $t$ and is at most $8 \om^3$ in expectation at $t_0(\om)$. The width of the 'end time interval' is $\om^2$. Apply the Markov inequality $\om^2$ times.}
put $\om'=\om^{5}$.
Vertex $u_1$ was chosen uniformly at random  from the occupied vertices in level one by the particle halting at $u_0$, so,
\begin{equation}\label{R1Tf}
\Pr( u_1 \in R_1(t)) \le  \frac{2\wh Z_1(t)}{L_1(t)} \le \frac{2 \om' \mu_0(t)}{\mu_1(t)} \le \frac{2 \om' t}{(k+1)(n/k)}
\le \frac{2 \om' t}{n}.
\end{equation}
Next consider the colour of $u_2$.
If $u_2$ is red at step $t$, then either  (i) it became red  before step $s_1$ when it was chosen uniformly at random by $u_1$, or (ii) it became red at some later step.

In the first case,
\[
\Pr(u_2 \in R_2(s_1))\le\frac{2\wh Z_2(s_1)}{L_2(s_1)} \le 2\om'\frac{\mu_1(s_1)}{\mu_2(s_1)}=\frac{2\om' s_1}{k(n/k)}=\frac{2 \om' s_1}{n}.
\]
By the gap theorem, $L_2$ is already concentrated at $s_1$, for otherwise $L_1(s_1)$ would be empty.
The  $\om'$ covers possible lack of concentration of $\wh Z_2$

In the second case, as $u_3$ is the chosen out-neighbour of $u_2$, then $u_2$ will turn red if some particle chooses $u_3$ after time $s_1$, and thus
\[
\Pr(u_2 \in R_2(t) \sm R_2(s_1)) \le \frac{1}{(n/k)} \sum_{\t=s_1+1}^t \E \ind{u_3 \text{ chosen at }\t}=\frac{t-s_1}{(n/k)}.
\]
Thus
\[
\Pr(u_2 \in R_2(t)) \le \frac{2\om' kt}{n},
\]
and similarly for $u_3,\ldots,u_{k-1}$. Thus
\[
\Pr(\text{path } vu_1\cdots u_k \text{ is blue}) \ge \prod_{i=1}^{k-1} \brac{1-\frac{2\om' kt}{n}}=1-O\bfrac{\om' k^2t}{n},
\]
where $t= t_f \le \om n^{k/(k+1)}$.
Provided $\om' k^2 \om/n^{1/(k+1)}=o(1)$, which holds for $k \le \sqrt{\log n}/\log \log n$, the path from the source to vertex $w=u_k$ level $k$ is blue, and thus the arborescence rooted at $w$ is exactly this path.
\end{proof}
%
%
%
%
%
\section{Analysis of DLA in the growing layers model} \label{Growing}
In the growing layers model, each layer is larger than the previous one by a factor of $d$. Thus $N_j=d^j$ for $j=0,1,...,k$ and we take $N_k=d^k=n$. Many of the properties of this model such as a gap property and unique connecting path are similar to the equal layers model.

The main, and most striking difference, is that there is a well defined level at a distance about $\sqrt{2k+2}$ from the end at which agglomerative growth stops and from which
and a single path grows back towards the source.
Moreover
at the end, {\em except for the last $...$
levels}, the connecting path is the {\em only occupied vertex} in the layer. 
This is in contrast to the equal layers model where  all levels have significantly occupancy, and even that of  the first level at the end is $\Th(n^{1/(k+1)})$. 

\begin{proposition}\label{endG-DLA}
Let $G$ be a growing layers graph with level sizes $N_i=d^i$, for $i=0,...,k$, where $d \rai,\; k \rai$,
and $k \ge \log^2 d$.
Let
\begin{equation}\label{GMTf}
T_f=  \sqrt{k}\; d^{k+3/2-\sqrt{2k+2}}.
\end{equation}
\begin{enumerate}[(1)]
\item  The finish time $t_f$ of DLA on $G$ satisfies  $ T_f/\om \le t_f \le \om T_f$, where $\om \rai$  slowly.
\item At $t_f$, levels $i=1,\ldots, k-\rdup{\sqrt{2k+2}-1}$ contain a single occupied vertex, the vertex $u_i$ of the the connecting path.
\end{enumerate}
\end{proposition}
\begin{proof}
The size $N_{i}$ of layer $i$ is  $N_i=d^{i}$. It follows that the product of the set sizes in
the denominator of $\mu_{k-j}(t)$ in \eqref{muk} is given by
\[
N_kN_{k-1}\cdots N_{k-j+1}= d^k d^{k-1}\cdots d^{k-j+1}= d^{kj-j(j-1)/2},
\]
and thus \eqref{muk} becomes
\begin{equation}\label{muk1}
\mu_{k-j}(t)=\frac{t^{j+1}}{(j+1)!\;d^k d^{k-1}\cdots d^{k-j+1}}= \frac{t^{j+1}}{(j+1)!\; d^{kj-j(j-1)/2}}.
\end{equation}
Assume $d$ is sufficiently large. The upper bound $\E \whL_{k-j}$ is obtained in Section \ref{Sec2}. However, a problem can arise in the growing layers model  in the upper bound calculations. The value of  $\mu_{k-j}$ can decrease with increasing $j$ and then (anomalously) increase again. This is  because the recurrence used to establish it assumes $\mu_{k-\ell} \rai$ for all $\ell <j$, which is not the case. We next locate where this happens; this is where the
unique path back to the source begins.

\paragraph{An important level.}
We next show the existence of a level $i\sim k+1-\sqrt{2k+2}$, such that the first occupancy of this level effectively determines the finish time of the process.

Let  $t$ be such that $\mu_{k-j}(t)= 1$, and let $t_1=t_1(k-j)$ be $\rdup{t}$, so that $\mu_{k-j}(t)\sim 1$ at step $t_1$.
From \eqref{muk1},
\begin{equation}\label{t1}
\mu_{k-j}(t_1)=\frac{t_1^{j+1}}{(j+1)!\; d^{kj-j(j-1)/2}}\sim 1 \quad \implies \quad t_1\sim [(j+1)!]^{1/(j+1)}\; d^{\frac{ 2kj-j(j-1)}{2(j+1)}}.
\end{equation}
From \eqref{muk}
\begin{equation}\label{cluck}
\frac{\mu_{k-j+1}(t)}{\mu_{k-j}(t)}=\frac{(j+1)d^{k-j+1}}{t},
\end{equation}
so setting $\mu_{k-j}(t_1(k-j))\sim  1$ gives
\begin{align}
\mu_{k-j+1}(t_1)\sim&\; \frac{j+1}{[(j+1)!]^{1/j+1}}\;d^{k-j+1-\frac{ 2kj-j(j-1)}{2(j+1)}}\nonumber\\
\sim &\;\frac{e}{(2\pi (j+1))^{1/2(j+1)}}\; d^{\frac{2k+2-j(j+1)}{2(j+1)}}.\label{muk-j+1}
\end{align}
The leading term on the RHS is  bounded, and the exponent of $d$ on the RHS is positive provided $2k+2 > j(j+1)$, which ensures that $\E \whL_{k-j+1}(t)$ is sufficiently large  close to $t_1(k-j)$.

What value of $k-j$ maximizes the step $t_1=t_1(k-j)$ at which  $\mu_{k-j}(t)\sim 1$?
Write the exponent of $d$ on the RHS of \eqref{t1} as $f(j)/2$ where
\[
f(j)=\frac{2kj-j(j-1)}{(j+1)}=(2k+2) -j -\frac{2k+2}{j+1}.
\]
 The maximum  of $f(j)$ occurs at $j^*$ when $(j^*+1)^2=(2k+2)$, giving $f(j^*)=(j^*)^2$.
The (not necessarily integer) values of $j^*$,  $k-j^*$ and $d^{f(j^*)/2}$  are
\begin{equation}\label{j*}
j^*=\sqrt{2k+2}-1, \qquad k-j^*=k+1-\sqrt{2k+2}, \qquad d^{f(j^*)/2}=d^{k+3/2-\sqrt{2k+2}}.
\end{equation}
In the case where $j^*$ is not integer, the
rounding error is addressed in the Appendix, where we show that the condition $k \ge \log^2 d$
given in Proposition \ref{endG-DLA} is sufficient
 to ignore the effect of rounding on the value of $T_f$.

Ignoring rounding effects, we evaluate $t_1=t_1(k-j)$ at $j=j^*$, where
$f(j)=j^2$, to find
\begin{align}
t_1  \sim&\; [(j+1)!]^{1/(j+1)}\; d^{f(j)/2}\nonumber\\
\sim&\;  e^{-1}(\sqrt{2\pi})^{1/(j+1)} \;(j+1)^{1+1/2(j+1)}\;\;d^{j^2/2}\nonumber\\
=&\;C_{k-j}\; \sqrt{2k+2}\;\; d^{k+3/2-\sqrt{2k+2}},\label{t1val}
\end{align}
on inserting the values from \eqref{j*}, and where $ C_{k-j}=e^{-1}(1+O(1/j))$.
Note that $t_1=\Th( T_f)$. From \eqref{cluck},
\[
\frac{\mu_{k-(j+1)}(t)}{\mu_{k-j}(t)}=\frac{t}{(j+2)d^{k-j}}
\]
so that at $t_1$,  for some $C,\,C'=\Th(1)$,
\begin{equation}\label{mukki}
\mu_{k-j^*}(t_1)\sim 1, \qquad \mu_{k-(j^*-1)}(t_1)=Cd^{1/2}, \qquad \mu_{k-(j^*+1)}(t_1)=C'd^{1/2}.
\end{equation}
At first this seems confusing, as  one might expect to have $ \mu_{k-(j^*-1)}(t_1)=o(1)$  by analogy with the equal layers model.  Assuming $d^{1/2} \rai$, level $k-j^*+1$ is the last level at which the condition $\mu_{k-j+1} \rai$  is valid  in the
 recurrence from $k-j+1$ to $k-j$; and is where the assumption in Proposition \ref{Prop1} breaks down.

\paragraph{Gap property of $\whL$ and a lower bound on $L$.}
Note that  from the definition of $t_1(k-j^*)$, at $t_1/\om$, we have $\mu_{k-j^*}(t_1/\om)=O(\om^{-(j^*+1)})$, where $j^* \sim \sqrt{2k} \rai$.
 Thus w.h.p. $\whL_{k-j^*}(t_1/\om)=o(1)$ in the upper process. Consequently all levels $i=k-\ell$, $\ell \ge j^*$,
have $\whL_i(t)=0$, w.h.p., for $t \le t_1/\om$.

In what follows we only consider indexes $k-\ell$ where $0 \le \ell \le k-j^*+1$.
By \eqref{j*} we have $k-j^*+1=k+2-\sqrt{2k+2}$.

We see from \eqref{mukki} that $\mu_{k-j^*+1}(t_1(k-j^*))=C'd^{1/2}$. Thus although $d \rai$ so that
$\whL_{k-j^*+1}$ will be concentrated around  $\mu_{k-(j^*+1)}(t_1)$, we cannot expect it to be as strong as in Lemma \ref{conc-whp}. Fortunately this will not matter as $k-j^*$ is the last level to which we apply the gap argument.
For $\ell \le j^*-1$ at $t_1(k-\ell)$, the value of $\whL_{k-\ell+1}$ obeys Lemma \ref{conc-whp}. In particular, it can be checked that $t_1(k-j^*+1)=\Th(1) \sqrt k d^{k+5/2-(3/2)\sqrt{2k+2}}$. Thus using \eqref{cluck}, we obtain $\mu_{k-j^*+2}(t_1(k-j^*+1)) =\Th(d^{(\sqrt{2k+2}+1)/2})$.

Turning to the lower bound $\wt L$ as given in \eqref{LLi}--\eqref{LLk} we need to prove that \eqref{small} holds.
The main task  is to find a value of $\b$ for the growing layers model which we can use in the
arguments given in  Section \ref{lbo} for the equal layers model. In what follows $\om\rai$ slowly. The value 
 from Lemma \ref{conc-whp} is denoted as $\om'=6 \log n$.

Using \eqref{GMTf}, \eqref{muk1}, and $N_{k-\ell}=d^{k-\ell}$,
for $\ell \ge 1$,
\begin{flalign*}
\frac{\mu_{k-\ell}}{N_{k-\ell}}=&\; \frac{t}{(k-\ell+1)d^{k-\ell}}\;\frac{\mu_{k-\ell+1}}{N_{k-\ell+1}}\\
\le&\; \frac{\mu_{k-\ell+1}}{N_{k-\ell+1}}\; \frac{\om \sqrt{k}}{(k+2-\sqrt{2k+2})} \;\frac{d^{k+3/2-\sqrt{2k+2}}}{d^{k+2-\sqrt{2k+2}}}\\
\le& \;\frac{\om (1+o(1))}{\sqrt{kd}}\; \frac{\mu_{k-\ell+1}}{N_{k-\ell+1}}\;
\le \; \bfrac{1}{\b}^\ell \frac{t}{N_k}.
\end{flalign*}
Choosing  $\b= (1+o(1)))\sqrt{kd}/\om$, \eqref{bound} becomes
\[
\sum_{\ell=j^*+1}^i \frac{L^*_{k-\ell}}{N_{k-\ell}} \le  \frac{O(\om'^3)}{N_{k-j^*}}+2 \sum_{\ell=j^*+1}^i
\bfrac{1}{\b}^\ell \frac{t}{N_k} =O \bfrac{t}{\b^i N_k}.
\]
It now follows from the proof in Section \ref{lbo} that for those $k-j^*+1 \le i \le k$, and $t \ge t_i(\om')$, then $\mu_i(t) \rai$  suitably fast. We hence obtain that $\E \wtL_i(t) \sim \mu_i(t) \sim \E \whL_i(t)$,
and so we have $L_{k-\ell}(t) \sim \whL_{k-\ell}(t)\sim \mu_{k-\ell}(t)$.

\subsection{The finish time  in the growing layers model.}\label{FGLM}
We now turn to the proof of \eqref{GMTf}.
We show that, w.h.p.,  a path (of occupied vertices) grows back to the source from the {\em first vertex to be occupied } in level $k-j^*$,
thus halting the process; and moreover this occurs {\em  before a second vertex becomes occupied} in level $k-j^*$.

Let $t_0$ be the first step at which $L_{k-j^*}(t) =1$, where w.h.p. $t_0 \ge t_1/\om$.
Then either $t_0 \le t_1$, or, as the probability
a particle halts at level $k-j^*$ is
\[
\phi=\frac{L_{k-(j^*-1)}(t_1)}{N_{k-(j^*-1)}} \sim\frac{C d^{1/2}}{d^{k-j^*+1}}=\frac{C}{d^{k+3/2-\sqrt{2k+2}}};
\]
the  probability this does not occur in a further $t_1$ steps is, see \eqref{t1val},
\[
(1-\f)^{t_1} \le e^{-t_1\f}= e^{-C'\sqrt{2k+2}}=o(1),
\]
where $C'\sim C/e$ and we assume $k \rai$.


Let $u$ be the vertex in level $k-j^*$ containing the unique  particle halted at $t_0$.
Construct a path back from $u$ to the source as follows. Wait  until a particle
halts at $w_{k-j^*-1}$ in level $k-j^*-1$ by choosing edge $w_{k-j^*-1}u$. The expected time for this is $d^{k-j^*}$. In a further expected time $d^{k-j^*-1}$, the path will extend backwards, as  a particle will halt in  level $k-j^*-2$ by choosing edge to $w_{k-j^*-1}$ etc. Thus in a further
\[
T=d^{k-j^*}+ d^{k-j^*-1}+ \cdots +d= d^{k-j^*} \bfrac{1-1/d^{k-j^*}}{1-1/d}= \Th(d^{k-j^*})=\Th(d^{k+1-\sqrt{2k+2}})
\]
expected steps there will be a path $vw_1\cdots w_{k-j^*-1}u$ of halted particles extending from the source $v$ to vertex $u$ thus stopping the DLA process (if it has not already halted). This path should be unique, as the expected time for it to branch backwards at any level $i$ is $d^i \gg d^{i-1}$ if $d \rai$.

We next give the proof of Proposition \ref{endG-DLA}.(2).
The expected number of steps needed to create another halted particle in level $k-j^*$ is
\[
\frac 1\f= \Th(d^{k+3/2-\sqrt{2k+2}})= \Th(T d^{1/2}).
\]
Whereas, w.h.p. on the assumption that $d \rai$,  in at most $(t_1 +T)\om$ steps the process has halted as claimed, 
before a second vertex can become occupied in level $k-j^*$.

\paragraph{Existence of a unique connecting path.}
Finally, we prove that the arborescence rooted at level $k$ containing the connecting path from source to sink, consists uniquely of that path. The proof is similar to  Lemma \ref{blue} for the equal layers model.  At $t_f$, w.h.p. there is a unique path from level $k-j^*$ to level zero, so that level $k-j^*+1$ plays the role of level one. By analogy with Section \ref{uniqp} equation \eqref{R1Tf} etc.,
\begin{equation}\nonumber
\Pr( u_{k-j^*+1} \in R_{k-j^*+1}(t)) = \frac{Z_{k-j^*+1}(t)}{L_{k-j^*+1}(t)} \le  \frac{\om}{\mu_{k-j^*+1}(t)} \le O\bfrac{\om^2}{d^{1/2}},
\end{equation}
where we used an earlier result that $\mu_{k-j^*+1}(t_1)=Cd^{1/2}$.

Thus as $t_f \le \om t_1(j^*)$, where $j^*=\sqrt{2k+2}-1$ and $t_1(k-j^*)$ is given by \eqref{t1val}
\[
\Pr( (v,u_k)\text{--path is blue}) \ge \brac{1-\frac{\om^2}{d^{1/2}}}\;\prod_{j=1}^{j^*-2} \brac{1-\frac{\om  t_f}{d^{k-j}}}=1-O\bfrac{\om^2}{d^{1/2}}-O\bfrac{\om \sqrt k}{d^{3/2}}.
\]
\end{proof}

\section{Theorem \ref{Th2}.(1): Trees with large branching factor}\label{555}

Let $G=G(k,d)$ be a labelled tree with branching factor $d$ and final level $k$, such that $d^k=n$. As before, the source of the particles is the unique vertex $v$ at level zero. An artificial sink vertex $z$
 (at level $k+1$) is  attached to the vertices at level $k$. 
To establish Theorem \ref{Th2}.(1), we need to prove w.h.p. that $T_f/\om \le t_f \le \om T_f$  where $\om\rai$ slowly with $n$.

We  borrow several ideas from the growing layers model, starting with $j^*$ and $T_f$ (see \eqref{j*} and  \eqref{GMTf}).
Let $j^*=\sqrt{2k+2}-1$ and let $T_f$ be given by
\begin{equation}\label{TfT}
T_f=  \sqrt{k}\; d^{k+3/2-\sqrt{2k+2}}.
\end{equation}

By a uniformity argument, the expected number of particles arriving at a given vertex in level $k-j^*$ by step $T_f$ is  $T_f/d^{(k-j^*)}=\sqrt{kd}$.
However, the expected number of particles arriving at  a vertex $w$ at level $k-j^*+1$ by step $T_f$ is
\begin{equation}\nonumber
\frac{T_f}{d^{k-j^*+1}}=  \sqrt{\frac{k}{d}}=o(1),
\end{equation}
provided $k \ll d$, which we assume to be true. In order for $w$ to be occupied,  the sub-tree rooted at $w$ must contain a path from $w$  to level $k$
consisting  of $j^*$ occupied vertices.
By  the Markov inequality the event that  $j^*=\Th(\sqrt k)$  particles have  arrived at $w$  by step $T_ f$, has probability $O(1/\sqrt d)$.
So there should be few vertices in level $k-j^*+1$ with a path  to level $k$ containing  $j^*$  halted particles.

Consider  $t$ particles  percolating downward 
from the root of an {\em infinite d-ary tree}. Particle $s$ starts at step $s$   and each particle transitions one edge at each step, so they never collide. Let $w$ be a vertex in level $\ell$ of this tree, where the root vertex is at level zero. Let $u$ be a vertex at level $\ell+j$ contained in the subtree rooted at $w$,  and $H(w,u)$ the unique path $w=w_0w_1...w_j=u$ from $w$ in level $\ell$ to $u$ in level $\ell+j$.
Let $1 \le s_0\le s_1 \le s_2 \le \cdots \le s_j \le t$ be  particle labels,  where particles $s_0$ and $s_1$ transition all edges of  $H$, particle $s_2$ transitions  edges  $w_0w_1...w_{j-1}$; in general $s_{i}$ transitions $w_0w_1...w_{j-i+1}$, and $s_{j}$ transitions $w_0w_1$. The probability of this is
\[
P(s_0,s_1,...,s_j)=\bfrac{1}{d^\ell}^{j+1} \frac 1{d^j} \frac 1{dd^2\cdots d^j}.
\]
The expected number of $j+1$ tuples of particles which transition as above is ${t \choose j+1}P(s_0,s_1,...,s_j)$.
There are $d^j$ vertices $u$ at level $j$ in the subtree of $w$, and $d^{\ell}$ vertices $w$ in level $\ell$ of the tree, so  the  number of sequences $Z_{\ell,j+1}(t)$ forming such a path between the levels  has expectation
\[
\E Z_{\ell,j+1}(t)=d^\ell d^j {t \choose j+1}P(s_0,s_1,...,s_j)={t \choose j+1} \frac{1}{d^{\ell+1} \cdots d^{\ell+j}}.
\]
For $j \ll t$, with $\ell=k-j$ and $\mu_{k-j}(t)$ as given by \eqref{muk1} for the growing layers model,
\begin{equation} \label{moomoo}
\E Z_{k-j,j+1}(t)\sim \frac{t^{j+1}}{(j+1)!} \frac{1}{d^{k-j+1} \cdots d^{k}}=\mu_{k-j}(t).
\end{equation}
Return now to the finite Cayley tree $G(k,d)$. For some vertex in level $k-j^*$ to be occupied at step $t$ there must be some sequence $(s_0,s_1,...,s_j)$ which satisfies the construction given above. Indeed $s_0$  halts in level $k$, causing $s_1$ to halt in level $k-1$ and so on, until $s_j$ halts in level $k-j$.

Let $j=j^*$ and $t=(1-\e)t_1(k-j^*)$ where $\e=\om/\sqrt{2 k+2}=\om/(j^*+1)$, and  $\om < \sqrt k$,
\begin{equation}\label{mutter}
\mu_{k-j^*}(t)={(1-\e)^{j^*+1}}\le e^{-\om}=o(1).
\end{equation}
 We conclude that at $t'$  no such sequence exists w.h.p. and 
levels $0,1,...,k-j^*$ are empty. By \eqref{t1val}, $t_1(k-j^*) =C T_f$, for some constant $C>0$, so the finish time $t_f \ge T_f/\om$.

From \eqref{cluck}, $\mu_{k-j+1}(t)=[(j+1)d^{k-j+1}/t] \cdot \mu_{k-j}(t)$, and so 
 \[
 \mu_{k-j^*+1}(t')=\frac{\sqrt{2k+2}\,d^{k-j^*+1}}{(1-\e)CT_f} \mu_{k-j^*}(t_1)= \Th(\sqrt d).
 \]
Returning briefly to the infinite $d$-regular tree process, let $Z'(t)$ be the number of vertices $w$ in level $k-j^*+1$ with $j^*$ particles following a path $w=w_1w_2...w_j=u$ in the subtree rooted at $w$, plus another particle which passes through $w$, to any  of its children. Then
\[
\E Z'(t) \le \frac{t}{d^{k-j^*+1}} \E Z_{k-j^*+1,j^*}(t) = \frac{t}{d^{k-j^*+1}}\mu_{k-j^*+1}(t).
\]
If $t= t_1(k-j^*)$,  then for $k \ll d$
\[
\E Z'(t_1)= \Th( \sqrt{k}) = o(\mu_{k-j^*+1}(t_1)).
\]
Thus in expectation there are $(1-o(1))  \mu_{k-j^*+1}(t_1)$ vertices $w$ in level $k-j^*+1$ are {\em exact}.
They have occupancy $j^*$,
in their subtree and a unique occupied path to level $k$. The events that two such vertices $w,w'$ have this path property are independent and we conclude that the number of such vertices is concentrated around its mean at $\Th(\sqrt d)$.

From now on the proof mirrors that of the growing layers model in Section \ref{FGLM}. Thus w.h.p. within a most
$2 t_1$ steps 
 the first  occupancy of a vertex in level $k-j^*$ has occured; and a  path grows back to the source from this vertex, halting the process  before the second occupancy in level $k-j^*$ can occur.

 \subsection{Theorem \ref{Th2}.(2):  DLA on trees with branching factor $d \ge 2$}
\begin{proposition}
For $d \ge 2$, let $G=G(k,d)$ be the Cayley
tree with branching factor
$d$ and height $k$, where  $d^k=n$. Let $T=T(G)$ be given by
\[
T=  \sqrt{k} d^{k+3/2-\sqrt{2k+2}}.
\]
Let $t_f$ be the finish time of DLA on $G$.
Then w.h.p.  $T/\om \le t_f \le \om T$.
\end{proposition}
\begin{proof}
The argument leading to \eqref{mutter} in the previous section only assumed $k \rai$ so that we could find some $\om<\sqrt k$ and where $\om \rai$. As $d^k=n$ then $k=\log n/(\log d)$ which is is monotone increasing with deceasing $d$. So $k \rai$ as before and the lower bound on $t_f$ follows from \eqref{mutter} on choosing $t'=(1-\e)t_1(k-j^*)$, and noting that  $t' \approx (1-\e) T/e > T/\om$.

For the upper bound, let $j=j^*-h$ where $h=\rdup{1/2+\log_d k}$. In Section \ref{555}, for the upper bound, we assumed that $k\ll d$, in which case $h=1$, which led to an argument about the occupancy of levels $k-j^*$ and $k-j^*+1$.
In this section the assumption $k \ll d$ may no longer hold,  for example, when $d=2$, $k=\log n/\log 2$.

The value of  $\mu_{k-j*}$  given by \eqref{moomoo}, satisfies $\mu_{k-j+1}(t)=[(j+1)d^{k-j+1}/t] \cdot \mu_{k-j}(t)$.
Iterating this, and evaluating at  $t_1(k-j^*)=(1+O(1/j))e^{-1}\sqrt{2k+2}d^{k-j^*+1/2}$ from \eqref{t1val},
\[
\frac{\mu_{k-j^*+h}(t)}{\mu_{k-j^*}(t)}=\frac{(j+1)\cdots(j+h)}{t^h}d^{k-j+1}\cdots d^{k-j+h}=\Th(1)\, e^{h}d^{h^2/2}.
\]
 This ratio is $\om(1)$,  as either (i) $d \rai$, or (ii) $d$ is constant and then $\log_d k=\log_d \log_d n$, implying that $h=\om(1)$.
As before, with $\E Z_{k-j^*+h,j^*-h+1}(t)=\mu_{k-j^*+h}(t)$, consider $Z'(t_1)$, where
\[
\E Z'(t_1(k-j^*)) \le \frac{t_1}{d^{k-j^*+h} }\mu_{k-j^*+h}(t_1).
\]
The ratio $\r=t_1(k-j^*)/d^{k-j^*+h}$ satisfies
\[
\r=t_1(k-j^*)/d^{k-j^*+h}=\Th(1) \frac{\sqrt{k}}{d^{h+1/2}}=\bfrac{k}{d^{2h+1}}^{1/2}.
\]
Assume $k \le d$, then $2h+1 \ge 3$, so $\r \le 1/d$. Alternatively, if $d \le k$, then $h\ge \log_d k$ and
$\r\le 1/k$. In either case  w.h.p. $(1-o(1))\mu_{k-j^*+h}(t)$ vertices in level $k-j^*+h$ are exact at around $t=t_1$.

Either some path of halted particles already extends to a level $i$ where $i<k-j^*+h$, or all vertices in these levels are unoccupied at $t_1$. 
In the latter case,  
w.h.p. there exists $\Th(e^hd^{h^2/2})$ exact paths from level $k-j^*+h$ to level $k$. 
In expectation, it takes at most
\[
\Th(1) \frac{d^{k-j^*+h}}{e^hd^{h^2/2}}\le T
\]
further  steps for one of these paths to extend  back to the source. The upper bound now follows from the Markov inequality.
\end{proof}

\paragraph{Proof of Theorem \ref{Th2}.(2).} If $d$ is constant, $d^k=n$ implies $k=\log n/\log d$. Hence  $\sqrt{2k+2}= \sqrt{2k}+ \Th(1)$, and $d^{3/2}=\Th(1)$. Thus  $t_1$, $T$ are both $\Th( T_f)$, where $T_f$ is as given.

\section*{Appendix}
\paragraph{The effect of rounding error in the growing layers model.}


We examine conditions on $d,k$ which allow us to effectively ignore the rounding error on $j^*$ in the value of $T_f$,
and show that $(\log d)/\sqrt k=O(1)$ suffices.
In the case that $j^*$ is not  integer, we require the maximum $j$ such that $2k+2 > j(j+1)$; see the exponent of $d$ in \eqref{muk-j+1}.
Clearly $j=\rdown{j^*}$ satisfies $2k+2 > j(j+1)$, but what about $j=\rdup{j^*}$? Put $j=j^*+\e$. Further analysis, not given here, shows that the condition $2k+2 \ge j(j+1)$, is  satisfied by $j=\rdup{j^*}$ up to some $\e \in (1/2,1)$.

Let $j=\max\{i: (2k+2) \ge i(i+1)\}$ and suppose that $j=\rdown{j^*}$ so that $j^*=j+\e$.
Let $T_M=\Th(t_1(k-j))$ be given by
\[
T_M=  \sqrt{k} d^{\frac 12 (2k+2 -j- (2k+2)/(j+1))},
\]
be a revised estimate of the order of the halting time, where $T_M \le T_f$ as $j^*$
maximizes $T_f$. As $j^*=\sqrt{2k+2}-1$,  $T_f=\sqrt{k}d^{({j^*}^2/2)}$, see \eqref{j*},  and $2k+2-j^*-(j^*)^2=\sqrt{2k+2}$,
\begin{align*}
\frac{T_M}{T_f}= & d^{\frac 12 \brac{ (2k+2)-j^*+\e -\frac{2k+2}{j^*+1-\e}-(j^*)^2\,}} \\
=& d^{\frac 12 \brac{\sqrt{2k+2}-\frac{\sqrt{2k+2}}{1-\e/(\sqrt{2k+2})} +\e}}\\
=& d^{-\frac{\e^2}{2\sqrt{2k+2}}(1+O(1/\sqrt{k}))}.
\end{align*}
Choosing  $j=\rdup{j^*}=j^*+\e$ gives the same result.
Thus the effect of rounding is to alter $T_f$ by $\Th(1)d^{-O(1/\sqrt{k})}$. Thus $(\log d)/\sqrt k=O(1)$ suffices.

\end{document}